\documentclass{amsart}
\usepackage{amsfonts,amssymb, wasysym}
\usepackage{multicol}
\usepackage{hyperref}
\hypersetup{colorlinks,citecolor=blue,linkcolor=blue}

\newtheorem{theorem}{Theorem}
\newtheorem{lemma}{Lemma}
\newtheorem{corollary}{Corollary}

\newcommand{\integers}{{\mathbb Z}}
\newcommand{\realnos}{{\mathbb R}}

\def\ov{\overline}

\def\Nu{{\rm N}}
\def\Mu{{\rm M}}

\def\Kappa{{\rm K}}

\def\Rho{{\rm P}}

\begin{document}

\title{Co-Seifert Fibrations of Compact Flat Orbifolds}

\author{John G. Ratcliffe and Steven T. Tschantz}

\address{Department of Mathematics, Vanderbilt University, Nashville, TN 37240
\vspace{.1in}}

\email{j.g.ratcliffe@vanderbilt.edu}

\date{}

\begin{abstract}
In this paper, we develop the theory for classifying all the geometric fibrations 
of compact, connected, flat $n$-orbifolds, over a 1-orbifold, up to affine equivalence. 
We apply our classification theory to classify all the geometric fibrations 
of compact, connected, flat $2$-orbifolds, over a 1-orbifold, up to affine equivalence. 
This paper is an essential part of our project to give a geometric proof of the classification 
of all closed flat 4-manifolds. 
\end{abstract}

\maketitle

\section{Introduction}\label{S:1} %1
An {\it $n$-dimensional crystallographic group} ({\it $n$-space group}) 
is a discrete group $\Gamma$ of isometries of Euclidean $n$-space $E^n$ 
whose orbit space $E^n/\Gamma$ is compact. 
If $\Gamma$ is an $n$-space group, then $E^n/\Gamma$ is a compact, connected, 
flat $n$-orbifold, and conversely if $M$ is a compact, connected, flat $n$-orbifold, 
then there is an $n$-space group $\Gamma$ such that $M$ is isometric to $E^n/\Gamma$. 
Henceforth, we assume that all orbifolds are connected unless otherwise stated. 
%If $\Gamma$ and ${\rm H}$ are $n$-space groups, then $\Gamma$ and ${\rm H}$ 
%are isomorphic if and only if $E^n/\Gamma$ and $E^n/{\rm H}$ are affinely equivalent. 

Informally, a {\it geometric fibration} of a flat $n$-orbifold $M$ over a flat $m$-orbifold $B$ is a surjective map $\eta: M \to B$, 
with totally geodesic fibers, that restricts to a fiber bundle over the ordinary set of the base orbifold $B$. 
For the formal definition of a geometric fibration, see our paper \cite{R-T}. 
A fiber of a geometric fibration $\eta: M \to B$ over an ordinary point of $B$ is called a {\it generic fiber}, and 
a fiber of $\eta: M \to B$ over a singular point of $B$ is called a {\it singular fiber}. 

In this paper, we develop the theory for classifying all the geometric fibrations of a compact 
flat $n$-orbifold, over a 1-orbifold, up to affine equivalence. 
A geometric fibration $\eta:M \to B$ over a 1-orbifold $B$ is called a {\it co-Seifert fibration}. 
There are two possibilities for a co-Seifert fibration $\eta: M \to B$: either $B$ is a circle or $B$ is a closed interval. 
If $B$ is a circle, then $\eta: M \to B$ is a fiber bundle projection, and the classification of such fibrations is easy and well known. 
If $B$ is a closed interval, then $\eta$ has precisely two singular fibers, and the classification of such fibrations is intricate and new. 

In our previous paper \cite{R-T}, we proved that a geometric fibration $\eta: M \to B$ 
corresponds to a space group extension 
$$1 \to \Nu \hookrightarrow \Gamma \to \Gamma/\Nu\to 1.$$
The problem of classifying co-Seifert fibrations $\eta: M \to B$, up to affine equivalence, is equivalent to classifying all pairs $(\Gamma,\Nu)$ such that $\Gamma$ is an $n$-space 
group and $\Nu$ is a normal subgroup of $\Gamma$, such that $\Gamma/\Nu$ is 
infinite cyclic or infinite dihedral, up to isomorphism. 
In our recent paper \cite{R-T-B}, we proved that for each dimension $n$, there are only finitely many isomorphism classes of such pairs. As an application of our theory, we describe the classification for $n = 2$.  
In a subsequent paper, we will describe the classification for $n = 3$.

We became interested in co-Seifert fibrations of compact flat orbifolds 
because every closed flat 4-manifold geometrically fibers over a 1-orbifold according to Hillman \cite{H}. 
In his paper \cite{H}, Hillman attempts to classify all closed flat 4-manifolds, up to affine equivalence, by listing 
all the possible co-Seifert fibrations of a closed flat 4-manifold. Hillman's argument is incomplete because 
he did not list all the possible co-Seifert fibrations over a closed interval. 
In this paper, we develop the theory necessary to find all the possible co-Seifert fibrations of a closed 
flat 4-manifold over a closed interval.  We will complete Hillman's geometric classification 
of closed flat 4-manifolds in a future paper. 

Although there are computer programs such as CARAT \cite{O} that classify crystallographic groups and 
compact flat manifolds, 
we have found in practice that the easiest way to identify a compact flat 3- or 4-manifold is by understanding 
how the manifold fibers. Knowing how a flat manifold fibers also reveals much about the geometry of the manifold. 
For example, the most salient feature of the geometry of the Hantzche-Wendt flat 3-manifold \cite{Z}
is that it geometrically fibers over a closed interval with generic fiber a torus and singular fibers Klein bottles.

Our paper is organized as follows:  In Sections \ref{S:2} - \ref{S:5}, we review definitions and basic 
results from our previous papers \cite{R-T, R-T-I, R-T-C, R-T-B} that are necessary to  
read this paper. The main results of our paper are in Sections \ref{S:7} and \ref{S:8}.  
In Section \ref{S:7}, we describe the classification of all pairs $(\Gamma, \Nu)$ such that 
$\Gamma$ is an $n$-space group and $\Nu$ is a normal subgroup of $\Gamma$, with $\Gamma/\Nu$ infinite cyclic, up to isomorphism.  In Section \ref{S:8}, we describe the classification 
of all pairs $(\Gamma, \Nu)$ such that $\Gamma/\Nu$ is infinite dihedral. 
In Section \ref{S:10}, we describe the classification of all the geometric fibrations of 
compact, connected, flat 2-orbifolds up to affine equivalence. 

\section{Complete Normal Subgroups}\label{S:2} % 2  

A map $\phi:E^n\to E^n$ is an isometry of $E^n$ 
if and only if there is an $a\in E^n$ and an $A\in {\rm O}(n)$ such that 
$\phi(x) = a + Ax$ for each $x$ in $E^n$. 
We shall write $\phi = a+ A$. 
In particular, every translation $\tau = a + I$ is an isometry of $E^n$. 

Let $\Gamma$ be an $n$-space group. 
Define $\eta:\Gamma \to {\rm O}(n)$ by $\eta(a+A) = A$. 
Then $\eta$ is a homomorphism whose kernel is the group $\mathrm{T}$ 
of translations in $\Gamma$. 
The image of $\eta$ is a finite group $\Pi$ called the {\it point group} of $\Gamma$.

Let $\Nu$ be a subgroup of an $n$-space group $\Gamma$. 
Define the {\it span} of $\Nu$ by the formula
$${\rm Span}(\Nu) = {\rm Span}\{a\in E^n:a+I\in \Nu\}.$$
Note that ${\rm Span}(\Nu)$ is a vector subspace $V$ of $E^n$.  
Let $V^\perp$ denote the orthogonal complement of $V$ in $E^n$. 

\begin{theorem} {\rm (Theorem 2 \cite{R-T})}\label{T:1} % 1
Let ${\rm N}$ be a normal subgroup of an $n$-space group $\Gamma$, 
and let $V = {\rm Span}(\Nu)$. 
\begin{enumerate}
\item If $b+B\in\Gamma$, then $BV=V$. 
\item If $a+A\in \Nu$,  then $a\in V$ and $ V^\perp\subseteq{\rm Fix}(A)$. 
\item The group $\Nu$ acts effectively on $V$ as a space group of isometries. 
\end{enumerate}
\end{theorem}

Let $\Gamma$ be an $n$-space group. 
The {\it dimension} of $\Gamma$ is $n$. 
If $\Nu$ is a normal subgroup of $\Gamma$, 
then $\Nu$ is a $m$-space group with $m= \mathrm{dim}(\mathrm{Span}(\Nu))$ by Theorem \ref{T:1}(3). 

\vspace{.1in}
\noindent{\bf Definition:}
Let $\Nu$ be a normal subgroup of an $n$-space group $\Gamma$, 
and let $V = {\rm Span}(\Nu)$.  
Then $\Nu$ is said to be a {\it complete normal subgroup} of $\Gamma$ if 
$$\Nu= \{a+A\in \Gamma: a\in V\ \hbox{and}\ V^\perp\subseteq{\rm Fix}(A)\}.$$

\begin{lemma} {\rm (Lemma 1 \cite{R-T})}\label{L:1} % 1
Let $\Nu$ be a complete normal subgroup of an $n$-space group $\Gamma$, 
and let $V={\rm Span}(\Nu)$. 
Then $\Gamma/\Nu$ acts effectively as a space group of isometries of $E^n/V$ 
by the formula
$({\rm N}(b+B))(V+x) = V+ b+Bx.$
\end{lemma}

\noindent{\bf Remark 1.} A normal subgroup $\Nu$ of a space group $\Gamma$ is complete 
precisely when $\Gamma/\Nu$ is a space group by Theorem 5 of \cite{R-T}. 
In particular, if $\Gamma/\Nu$ is infinite cyclic or infinite dihedral, then $\Nu$ 
is a complete normal subgroup of $\Gamma$.

\begin{theorem}{\rm (Theorem 4 \cite{R-T})}\label{T:2} % 2
Let ${\rm N}$ be a complete normal subgroup of an $n$-space group $\Gamma$, 
and let $V = {\rm Span}({\rm N})$.  
Then the flat orbifold $E^n/\Gamma$ geometrically fibers over the flat orbifold 
$(E^n/V)/(\Gamma/{\rm N})$ with generic fiber the flat orbifold $V/{\rm N}$ and 
fibration projection $\eta_V: E^n/\Gamma \to (E^n/V)/(\Gamma/{\rm N})$ defined by the formula 
$\eta_V(\Gamma x) = (\Gamma/\Nu)(V+x).$
\end{theorem}

\section{The Generalized Calabi Construction}\label{S:3} % 3

Let $\Nu$ be a complete normal subgroup of an $n$-space group $\Gamma$, 
let $V = \mathrm{Span}(\Nu)$, and let $V^\perp$ be the orthogonal complement of $V$ in $E^n$.  
Let $\gamma \in \Gamma$.   Then $\gamma = b+B$ with $b\in E^n$ and $B\in \mathrm{O}(n)$. 
Write $b = \overline b + b'$ with $\overline b \in V$ and $b' \in V^\perp$. 
Let $\overline B$ and $B'$ be the orthogonal transformations of $V$ and $V^\perp$, respectively, 
obtained by restricting $B$. 
Let $\overline \gamma = \overline b + \overline B$ and $\gamma' = b' + B'$. 
Then $\overline \gamma$ and $\gamma'$ are isometries of $V$ and $V^\perp$, respectively.

Euclidean $n$-space $E^n$ decomposes as the Cartesian product $E^n = V \times V^\perp$. 
Let $x\in E^n$.  Write $x = v+w$ with $v\in V$ and $w\in V^\perp$. Then 
$$(b+B)x = b+Bx = \overline{b}+b' + Bv + Bw = (\overline{b}+\overline{B}v) + (b'+B'w).$$
Hence, the action of $\Gamma$ on $E^n$ corresponds to the diagonal action of $\Gamma$ 
on $V\times V^\perp$ defined by the formula
$$\gamma(v,w) = (\overline{\gamma}v,\gamma'w).$$
Here $\Gamma$ acts on both $V$ and $V^\perp$ via isometries. 
The kernel of the corresponding homomorphism from $\Gamma$ to $\mathrm{Isom}(V)$ 
is the group
$$\Kappa = \{b+B\in\Gamma: b \in V^\perp\ \hbox{and}\ V \subseteq \mathrm{Fix}(B)\}.$$
We call $\Kappa$ the {\it kernel of the action} of $\Gamma$ on $V$. 
The group $\Kappa$ is a normal subgroup of $\Gamma$. 
The action of $\Gamma$ on $V$ induces an effective action of $\Gamma/\Kappa$ on $V$ via isometries.  
Note that $\Nu\cap\Kappa =\{I\}$, and each element of $\Nu$ commutes with each 
element of $\Kappa$. 
Hence $\Nu\Kappa$ is a normal subgroup of $\Gamma$, and $\Nu\Kappa$ is the 
direct product of $\Nu$ and $\Kappa$. 
The group $\Gamma/\Kappa$ acts on $V$ as a discrete group of isometries if and only if 
$\Gamma/\Nu\Kappa$ is a finite group by Theorem 3(4) of \cite{R-T-C}. 

The group $\Nu$ is the kernel of the action of $\Gamma$ on $V^\perp$, and so the action 
of $\Gamma$ on $V^\perp$ induces  an effective action of $\Gamma/\Nu$ on $V^\perp$ via isometries. 
Orthogonal projection from $E^n$ to $V^\perp$ induces an isometry from $E^n/V$ to $V^\perp$. 
Hence $\Gamma/\Nu$ acts on $V^\perp$ as a space group of isometries by Lemma \ref{L:1}. 

Let $\ov \Gamma = \{\ov \gamma: \gamma \in \Gamma\}$. 
If $\gamma \in \Gamma$, then $(\overline{\gamma})^{-1} = \overline{\gamma^{-1}}$, and 
if $\gamma_1, \gamma_2 \in \Gamma$, 
then $\overline{\gamma_1}\,\overline{\gamma_2}= \overline{\gamma_1\gamma_2}$.   
Hence $\ov \Gamma$ is a subgroup of $\mathrm{Isom}(V)$. 
The map $\Upsilon: \Gamma \to \ov \Gamma$ defined by $\Upsilon(\gamma)= \ov\gamma$ is an epimorphism with kernel $\Kappa$. 
The group $\ov\Gamma$ is a discrete subgroup of $\mathrm{Isom}(V)$ 
if and only if $\Gamma/\Nu\Kappa$ is finite by Theorem 3(4) of \cite{R-T-C}. 

Let $\Gamma' = \{\gamma' : \gamma \in \Gamma\}$.  
If $\gamma \in \Gamma$, then $(\gamma')^{-1} = (\gamma^{-1})'$, and 
if $\gamma_1, \gamma_2 \in \Gamma$, 
then $\gamma_1'\gamma_2' = (\gamma_1\gamma_2)'$.   
Hence $\Gamma'$ is a subgroup of $\mathrm{Isom}(V^\perp)$. 
The map $\Rho' : \Gamma \to \Gamma'$ 
defined by $\Rho'(\gamma) = \gamma'$ is epimorphism with kernel $\Nu$, 
since $\Nu$ is a complete normal subgroup of $\Gamma$.  
Hence $\Rho'$ induces an isomorphism $\Rho: \Gamma/\Nu \to \Gamma'$ defined by $\Rho(\Nu\gamma) = \gamma'$. 
The group $\Gamma'$ is a space group of isometries of $V^\perp$ 
with $V^\perp/\Gamma' = V^\perp/(\Gamma/\Nu)$. 

Let $\ov \Nu = \{\ov \nu: \nu \in \Nu\}$.  
Then $\ov \Nu$ is a subgroup of $\mathrm{Isom}(V)$. 
The map $\Upsilon: \Nu \to \ov \Nu$ defined by $\Upsilon(\nu) = \ov\nu$ is an isomorphism. 
The group $\ov\Nu$ is a space group of isometries of $V$ with $V/\ov{\Nu} = V/\Nu$. 

The action of $\Gamma$ on $V$ induces an action of $\Gamma/\Nu$ on $V/\Nu$ 
defined by 
$$(\Nu\gamma)(\Nu v) = \Nu \overline{\gamma}v.$$ 
The action of $\Gamma/\Nu$ on $V/\Nu$ determines a homomorphism 
$$\Xi : \Gamma/\Nu \to \mathrm{Isom}(V/\Nu)$$
defined by $\Xi(\Nu \gamma) = \overline \gamma_\star$, where $\overline \gamma_\star: V/\Nu \to V/\Nu$ 
is defined by $\overline\gamma_\star(\Nu v) = \Nu \overline\gamma(v)$.

The action of $\Nu$ on $V^\perp$ is trivial and the action of $\Kappa$ on $V$ is trivial. 
Hence $E^n/\Nu\Kappa$ decomposes as the Cartesian product 
$E^n/\Nu\Kappa = V/\Nu \times V^\perp/\Kappa.$

The action of $\Gamma/\Nu\Kappa$ on $E^n/\Nu\Kappa$ corresponds 
to the diagonal action of $\Gamma/\Nu\Kappa$ on $V/\Nu \times V^\perp/\Kappa$ via isometries 
defined by the formula
$$(\Nu\Kappa(b+B))(\Nu v,\Kappa w) = (\Nu(c+Bv),\Kappa(d+Bw)).$$
Hence, we have the following theorem.
\begin{theorem}\label{T:3} % 3
{\rm (The Generalized Calabi Construction)}
Let $\Nu$ be a complete normal subgroup of an $n$-space group $\Gamma$, 
and let $\Kappa$ be the kernel of the action of $\Gamma$ on $V= \mathrm{Span}(\Nu)$. 
Then the map
$$\chi: E^n/\Gamma \to (V/\Nu\times V^\perp/\Kappa)/(\Gamma/\Nu\Kappa)$$
defined by $\chi(\Gamma x) = (\Gamma/\Nu\Kappa)(\Nu v, \Kappa w)$, 
with  $x = v + w$ and $v\in V$ and $w\in V^\perp$, is an isometry. 
\end{theorem}

We call $\Gamma/\Nu\Kappa$ the {\it structure group} 
of the geometric fibered orbifold structure on $E^n/\Gamma$ 
determined by the complete normal subgroup $\Nu$ of $\Gamma$ as in Theorem 2. 

The natural projection from $V/\Nu\times V^\perp/\Kappa$ to $V^\perp/\Kappa$ 
induces a continuous surjection
$$\pi^\perp: (V/\Nu\times V^\perp/\Kappa)/(\Gamma/\Nu\Kappa) \to V^\perp/(\Gamma/\Nu).$$
Orthogonal projection from $E^n$ to $V^\perp$ induces an isometry from $E^n/V$ to $V^\perp$ 
which in turn induces an isometry 
$$\psi^\perp: (E^n/V)/(\Gamma/\Nu) \to V^\perp/(\Gamma/\Nu).$$
\begin{theorem}\label{T:4} % 4
{\rm (Theorem 5 \cite{R-T-C})}
The following diagram commutes
\[\begin{array}{ccc}
E^n/\Gamma & {\buildrel \chi\over\longrightarrow} &
(V/\Nu\times V^\perp/\Kappa)/(\Gamma/\Nu\Kappa) \\
\eta_V \downarrow \  & & \downarrow\pi^\perp \\
(E^n/V)/(\Gamma/\Nu) & {\buildrel \psi^\perp\over\longrightarrow}  & V^\perp/(\Gamma/\Nu). 
\end{array}\] 
\end{theorem}

Theorem \ref{T:4} says that the geometric fibration $\eta_V$ is equivalent 
to the projection $\pi^\perp$ induced by the projection on the second factor 
of $V/\Nu \times V^\perp/\Kappa$. 
Note that the base of the geometric fibration $\pi^\perp$ 
is the orbit space of the action of the structure group $\Gamma/\Nu\Kappa$ on the second factor $V^\perp/\Kappa$, that is, 
$$V^\perp/(\Gamma/\Nu) = ( V^\perp/\Kappa)/(\Gamma/\Nu\Kappa).$$

%\vspace{.15in}
\noindent{\bf Remark 2.} The right-side of the diagram in Theorem \ref{T:4} gives a concise way of describing a geometric fibration determined by a complete normal subgroup $\Nu$ of a space group $\Gamma$ in terms of the generalized Calabi construction. 

\vspace{.15in}
If $\mathrm{Span}(\Kappa) = V^\perp$, then $\Kappa$ is a complete normal subgroup 
of $\Gamma$ called the {\it orthogonal dual} of $\Nu$ in $\Gamma$, 
and we write $\Kappa = \Nu^\perp$. 
By Theorem 6 of \cite{R-T-C}, the complete normal subgroup $\Nu$ of $\Gamma$ has an orthogonal dual if and only if 
the structure group $\Gamma/\Nu\Kappa$ is finite. 

Suppose $\Nu$ has an orthogonal dual, and so $\Kappa = \Nu^\perp$.  
The natural projection from $V/\Nu\times V^\perp/\Kappa$ to $V/\Nu$ 
induces a continuous surjection
$$\pi: (V/\Nu\times V^\perp/\Kappa)/(\Gamma/\Nu\Kappa) \to V/(\Gamma/\Kappa).$$
Orthogonal projection from $E^n$ to $V$ induces an isometry from $E^n/V^\perp$ to $V$ 
which in turn induces an isometry 
$$\psi: (E^n/V^\perp)/(\Gamma/\Kappa) \to V/(\Gamma/\Kappa).$$
The next corollary follows from Theorem \ref{T:4} by reversing the roles of $\Nu$ and $\Kappa$. 
\begin{corollary}\label{C:1} % 1
The following diagram commutes
\[\begin{array}{ccc}
E^n/\Gamma & {\buildrel \chi\over\longrightarrow} &
(V/\Nu\times V^\perp/\Kappa)/(\Gamma/\Nu\Kappa) \\
\eta_{V ^\perp}\downarrow\ \ \  & & \downarrow\pi \\
(E^n/V^\perp)/(\Gamma/\Kappa) & {\buildrel \psi\over\longrightarrow}  & V/(\Gamma/\Kappa). 
\end{array}\] 
\end{corollary}

Corollary \ref{C:1} says that the orthogonally dual geometric fibration 
$\eta_{V^\perp}$ is equivalent to the projection $\pi$ 
induced by the projection on the first factor of $V/\Nu\times V^\perp/\Kappa$. 
Note that the base of the geometric fibration $\pi$ is the orbit space of the action 
of the structure group $\Gamma/\Nu\Kappa$ on the first factor $V/\Nu$, that is, 
$$V/(\Gamma/\Kappa) = (V/\Nu)/(\Gamma/\Nu\Kappa).$$

\section{Affinities}\label{S:4}  % 4

A map $\alpha: E^n \to E^n$ is an {\it affinity} if and only if there is an $a\in E^n$ 
and an $A \in {\rm GL}(n,\realnos)$ such that $\alpha(x) = a + Ax$ for each $x\in E^n$, 
in which case, we write $\alpha = a + A$. 
Note that an affinity $\alpha = a + A$ of $E^n$ is an isometry of $E^n$ 
precisely when $A \in {\rm O}(n)$. 
The set ${\rm Aff}(E^n)$ of all affinities of $E^n$ is a group that contains 
the group ${\rm Isom}(E^n)$ of isometries of $E^n$ as a subgroup.  

If $\Gamma$ and ${\rm H}$ are $n$-space groups, 
then a map $\alpha: E^n/\Gamma \to E^n/{\rm H}$ is an {\it affinity} if and only $\alpha$ lifts 
to an affinity of $E^n$, that is, there is an affinity $\tilde \alpha$ of $E^n$ such that 
$\alpha(\Gamma x) = {\rm H}\tilde\alpha(x)$ for each $x \in E^n$   
from which we have that $\tilde\alpha\Gamma\tilde\alpha^{-1} = \mathrm{H}$. 
Moreover if $\beta:E^n \to E^n$ is an affinity 
such that $\beta\Gamma\beta^{-1} = \mathrm{H}$, 
then $\beta$ induces an affinity $\beta_\star: E^n/\Gamma \to E^n/\mathrm{H}$
defined by $\beta_\star(\Gamma x) = \mathrm{H}\beta(x)$ for each $x \in E^n$. 
Every affinity $\alpha: E^n/\Gamma \to E^n/{\rm H}$ is a homeomorphism whose 
inverse is also an affinity. 

Two flat $n$-orbifolds $E^n/\Gamma$ and $E^n/{\rm H}$ are said to be 
{\it affinely equivalent} if there is an affinity $\alpha: E^n/\Gamma \to E^n/{\rm H}$. 
If $\xi: \Gamma \to \mathrm{H}$ is an isomorphism, 
there is an affinity $\alpha$ 
of $E^n$ such that $\xi(\gamma) = \alpha\gamma\alpha^{-1}$ by Bieberbach's theorem. 
Hence, two $n$-space groups $\Gamma$ and ${\rm H}$ are isomorphic if and only 
if $E^n/\Gamma$ and $E^n/{\rm H}$ are affinely equivalent.

Let $\Gamma$ be an $n$-space group. 
Then the set $\mathrm{Aff}(E^n/\Gamma)$ of affinities of $E^n/\Gamma$ is a group.  
Let $\alpha$ be an affinity of $E^n/\Gamma$.  Then $\alpha$ lifts to an affinity $\tilde\alpha$ of $E^n$ such that $\tilde\alpha\Gamma\tilde\alpha^{-1} = \Gamma$.  
The affinity $\tilde\alpha$ of $E^n$ determines an automorphism $\tilde\alpha_\ast$ of $\Gamma$ defined by $\tilde\alpha_\ast(\gamma) = \tilde\alpha\gamma\tilde\alpha^{-1}$. 
If $\tilde\alpha'$ is another lift of $\alpha$, then $\tilde\alpha'= \gamma\tilde\alpha$ 
for some $\gamma$ in $\Gamma$. 
Hence, the element $\mathrm{Inn}(\Gamma)\tilde\alpha_\ast$ of $\mathrm{Out}(\Gamma)$ depends only on $\alpha$, and we have an homomorphism
$$\Omega: \mathrm{Aff}(E^n/\Gamma) \to \mathrm{Out}(\Gamma)$$
defined by $\Omega(\alpha)  = \mathrm{Inn}(\Gamma) \tilde\alpha_\ast$. 
The {\it Euclidean outer automorphism group} of $\Gamma$ is defined to be 
$$\mathrm{Out}_E(\Gamma) = \Omega(\mathrm{Isom}(E^n/\Gamma)).$$
By Theorem 2 of \cite{R-T-I}, the group $\mathrm{Out}_E(\Gamma)$ is finite. 

Let $\Gamma_1$ and $\Gamma_2$ be $n$-space groups, and let $\phi\in\mathrm{Aff}(E^n)$ 
such that $\phi\Gamma_1\phi^{-1} = \Gamma_2$. 
Then $\phi$ induces an affinity $\phi_\star:E^n/\Gamma_1 \to E^n/\Gamma_2$ 
defined by $\phi_\star(\Gamma_1x) = \Gamma_2\phi(x)$ for each $x\in E^n$. 
Define 
$$\phi_\sharp:\mathrm{Aff}(E^n/\Gamma_1) \to \mathrm{Aff}(E^n/\Gamma_2)$$
by $\phi_\sharp(\alpha) = \phi_\star\alpha\phi_\star^{-1}$. 
Then $\phi_\sharp$ is an isomorphism with $(\phi_\sharp)^{-1}=(\phi^{-1})_\sharp$. 

Let $\phi_\ast: \Gamma_1 \to \Gamma_2$ be the isomorphism defined 
by $\phi_\ast(\gamma) =\phi\gamma\phi^{-1}$. 
Define 
$$\phi_\#:\mathrm{Out}(\Gamma_1) \to \mathrm{Out}(\Gamma_2)$$
by $\phi_\#(\mathrm{Inn}(\Gamma_1)\zeta) = \mathrm{Inn}(\Gamma_2)\phi_\ast\zeta\phi_\ast^{-1}$.  
Then $\phi_\#$ is an isomorphism with $(\phi_\#)^{-1} = (\phi^{-1})_\#$. 

\begin{lemma}{\rm (Lemma 10 \cite{R-T-I})}\label{L:2} % 2
Let $\Gamma_1$ and $\Gamma_2$ be $n$-space groups, 
and let $\phi\in \mathrm{Aff}(E^n)$ such that $\phi\Gamma_1\phi^{-1} = \Gamma_2$. 
Then the following diagram commutes

$$
\begin{array}{ccc}
\mathrm{Aff}(E^n/\Gamma_1)  & {\buildrel \phi_\sharp \over \longrightarrow} & \mathrm{Aff}(E^n/\Gamma_2) \vspace{.05in} \\ 
\Omega\downarrow  &           & \downarrow \Omega \\
\mathrm{Out}(\Gamma_1) &  {\buildrel \phi_\# \over \longrightarrow}  & \mathrm{Out}(\Gamma_2)
\end{array}
$$
\end{lemma}

\section{Isomorphism classes}\label{S:5} % 5

In this section, we recall some of the definitions and results from \S 4 of \cite{R-T-B}. 
Let $m$ be a positive integer less than $n$. 
Let $\Mu$ be an $m$-space group, and let $\Delta$ be an $(n-m)$-space group.

\vspace{.15in}
\noindent{\bf Definition:} Define $\mathrm{Iso}(\Delta,\Mu)$ to be the set of isomorphism classes 
of pairs $(\Gamma, \Nu)$ 
where $\Nu$ is a normal subgroup of an $n$-space group $\Gamma$ 
such that $\Nu$ is isomorphic to $\Mu$ and $\Gamma/\Nu$ is isomorphic to $\Delta$. 
We denote the isomorphism class of a pair $(\Gamma,\Nu)$ by $[\Gamma,\Nu]$. 

\vspace{.15in}
\noindent{\bf Definition:} 
Define $\mathrm{Hom}_f(\Delta,\mathrm{Out}(\Mu))$ to be the set of all homomorphisms 
from $\Delta$ to $\mathrm{Out}(\Mu)$ that have finite image. 

\vspace{.15in}

%\newpage
The group $\mathrm{Out}(\Mu)$ acts on the left of $\mathrm{Hom}_f(\Delta,\mathrm{Out}(\Mu))$ by conjugation, 
that is, if $g\in \mathrm{Out}(\Mu)$ and $\eta\in \mathrm{Hom}_f(\Delta,\mathrm{Out}(\Mu))$, 
then $g \eta = g_\ast\eta$ where $g_\ast: \mathrm{Out}(\Mu) \to \mathrm{Out}(\Mu)$ is defined by $g_\ast(h) = ghg^{-1}$. 
Let $\mathrm{Out}(\Mu)\backslash\mathrm{Hom}_f(\Delta,\mathrm{Out}(\Mu))$ be the set of $\mathrm{Out}(\Mu)$-orbits. 

As explained in \S 4 of \cite{R-T-B}, 
the group $\mathrm{Out}(\Delta)$ acts on the right of the set 
$$\mathrm{Out}(\Mu)\backslash\mathrm{Hom}_f(\Delta,\mathrm{Out}(\Mu))$$ 
by
$$(\mathrm{Out}(\Mu)\eta)(\beta\mathrm{Inn}(\Delta)) = \mathrm{Out}(\Mu)(\eta\beta).$$
\noindent{\bf Definition:} Define the set $\mathrm{Out}(\Delta,\Mu)$ by the formula
$$\mathrm{Out}(\Delta,\Mu) = (\mathrm{Out}(\Mu)\backslash\mathrm{Hom}_f(\Delta,\mathrm{Out}(\Mu)))/\mathrm{Out}(\Delta).$$ 
If $\eta\in \mathrm{Hom}_f(\Delta,\mathrm{Out}(\Mu))$, let $[\eta] = (\mathrm{Out}(\Mu)\eta)\mathrm{Out}(\Delta)$ 
be the element of $\mathrm{Out}(\Delta,\Mu)$ determined by $\eta$. 
The set $\mathrm{Out}(\Delta,\Mu)$ is finite by Lemma 4.5 of \cite{R-T-B}.

\vspace{.15in}
Let $(\Gamma,\Nu)$ be a pair such that $[\Gamma,\Nu]\in\mathrm{Iso}(\Delta,\Mu)$. 
The action of $\Gamma$ on $\Nu$ by conjugation induces a homomorphism 
$$\mathcal{O}: \Gamma/\Nu \to \mathrm{Out}_E(\Nu)$$
defined by $\mathcal{O}(\Nu\gamma) = \gamma_\ast\mathrm{Inn}(\Nu)$ 
where $\gamma_\ast(\nu) = \gamma\nu\gamma^{-1}$ 
for each $\gamma \in \Gamma$ and $\nu\in\Nu$. 
If  $\alpha: \Nu \to \Mu$ is an isomorphism, then $\alpha$ induces an isomorphism
$$\alpha_\#: \mathrm{Out}(\Nu) \to \mathrm{Out}(\Mu)$$
defined by $\alpha_\#(\zeta\mathrm{Inn}(\Nu) )= \alpha\zeta\alpha^{-1}\mathrm{Inn}(\Mu)$ for each $\zeta\in\mathrm{Aut}(\Nu)$. 

Let $\alpha: \Nu \to \Mu$ and $\beta:\Delta \to \Gamma/\Nu$ be isomorphisms. 
Then we have that $\alpha_\#\mathcal{O}\beta \in \mathrm{Hom}_f(\Delta,\mathrm{Out}(\Mu))$. 
As explained in \S 4 of \cite{R-T-B}, there is a well-defined function 
$$\omega:\mathrm{Iso}(\Delta,\Mu) \to \mathrm{Out}(\Delta,\Mu)$$
defined by $\omega([\Gamma,\Nu]) =[\alpha_\#\mathcal{O}\beta]$. 

\section{Affinity Classes}\label{S:6} % 6

Let $m$ be a positive integer less than $n$. 
Let $\Mu$ be an $m$-space group, and let $\Delta$ be an $(n-m)$-space group. 
For simplicity, define $\mathrm{Aff}(\Mu) = \mathrm{Aff}(E^m/\Mu)$. 

\vspace{.15in}
\noindent{\bf Definition:} Define $\mathrm{Hom}_f(\Delta,\mathrm{Aff}(\Mu))$ to be the set of all homomorphism from $\Delta$ to $\mathrm{Aff}(\Mu)$ such that the composition 
with $\Omega:\mathrm{Aff}(\Mu)\to \mathrm{Out}(\Mu)$ has finite image.

\vspace{.15in}
By Lemma \ref{L:2}, the group $\mathrm{Aff}(\Mu)$ acts on the left of $\mathrm{Hom}_f(\Delta,\mathrm{Aff}(\Mu))$ by conjugation, 
that is, if $\alpha\in \mathrm{Aff}(\Mu)$ and $\eta\in \mathrm{Hom}_f(\Delta,\mathrm{Aff}(\Mu))$, 
then $\alpha \eta = \alpha_\ast\eta$ where $\alpha_\ast: \mathrm{Aff}(\Mu) \to \mathrm{Aff}(\Mu)$ is defined 
by $\alpha_\ast(\beta) = \alpha\beta\alpha^{-1}$. 
Let $\mathrm{Aff}(\Mu)\backslash\mathrm{Hom}_f(\Delta,\mathrm{Aff}(\Mu))$ be the set of $\mathrm{Aff}(\Mu)$-orbits. 

The group $\mathrm{Aut}(\Delta)$ acts on the right of $\mathrm{Hom}_f(\Delta,\mathrm{Aff}(\Mu))$ 
by composition of homomorphisms. 
If $\zeta\in \mathrm{Aut}(\Delta)$ and  $\eta\in \mathrm{Hom}_f(\Delta,\mathrm{Aff}(\Mu))$ and $\alpha\in \mathrm{Aff}(\Mu)$, 
then 
$$(\alpha\eta)\zeta = (\alpha_\ast\eta)\zeta = \alpha_\ast(\eta\zeta) = \alpha(\eta\zeta).$$
Hence $\mathrm{Aut}(\Delta)$ acts on the right of 
$\mathrm{Aff}(\Mu)\backslash\mathrm{Hom}_f(\Delta,\mathrm{Aff}(\Mu))$ 
by 
$$(\mathrm{Aff}(\Mu)\eta)\zeta = \mathrm{Aff}(\Mu)(\eta\zeta).$$
Let $\delta, \epsilon \in \Delta$ and $\eta\in \mathrm{Hom}_f(\Delta,\mathrm{Aff}(\Mu))$.  
Then we have that
$$\eta\delta_\ast(\epsilon) = \eta(\delta\epsilon\delta^{-1}) =\eta(\delta)\eta(\epsilon)\eta(\delta)^{-1}= \eta(\delta)_\ast\eta(\epsilon) 
= (\eta(\delta)\eta)(\epsilon). $$
Hence $\eta\delta_\ast = \eta(\delta)\eta$.  Therefore $\mathrm{Inn}(\Delta)$ acts trivially on
 $\mathrm{Aff}(\Mu)\backslash\mathrm{Hom}_f(\Delta,\mathrm{Aff}(\Mu))$. 
Hence $\mathrm{Out}(\Delta)$ acts on the right of 
$\mathrm{Aff}(\Mu)\backslash\mathrm{Hom}_f(\Delta,\mathrm{Aff}(\Mu))$ 
by
$$(\mathrm{Aff}(\Mu)\eta)(\zeta\mathrm{Inn}(\Delta))= \mathrm{Aff}(\Mu)(\eta\zeta).$$

\noindent{\bf Definition:} Define the set  $\mathrm{Aff}(\Delta,\Mu)$ by the formula
$$\mathrm{Aff}(\Delta,\Mu) = (\mathrm{Aff}(\Mu)\backslash\mathrm{Hom}_f(\Delta,\mathrm{Aff}(\Mu)))/\mathrm{Out}(\Delta).$$ 
If $\eta\in \mathrm{Hom}_f(\Delta,\mathrm{Aff}(\Mu))$, let $[\eta] = (\mathrm{Aff}(\Mu)\eta)\mathrm{Out}(\Delta)$ 
be the element of $\mathrm{Aff}(\Delta,\Mu)$ determined by $\eta$.

Let $\eta \in \mathrm{Hom}_f(\Delta,\mathrm{Aff}(\Mu))$. 
By Theorem 5 of \cite{R-T-I}, there exists $C \in \mathrm{GL}(m,\realnos)$ such that $C\Mu C^{-1}$ is an $m$-space group 
and $C_\#(\Omega\eta(\Delta)) \subseteq \mathrm{Out}_E(C\Mu C^{-1})$. 
By Lemma \ref{L:2}, we have that $\Omega C_\sharp\eta(\Delta) \subseteq \mathrm{Out}_E(C\Mu C^{-1})$. 
By Theorems 2 and 3 of \cite{R-T-I}, we deduce that $C_\sharp\eta(\Delta) \subseteq \mathrm{Isom}(E^m/C\Mu C^{-1})$. 
Extend $C\Mu C^{-1}$ to a subgroup $\Nu$ of $\mathrm{Isom}(E^n)$ such that the point group of $\Nu$ 
acts trivially on $(E^m)^\perp$. 
By Theorem 2.9 of \cite{R-T-B}, 
there exists an $n$-space group $\Gamma$ containing $\Nu$ as a complete normal subgroup 
such that $\Gamma' = \Delta$ and if $\Xi: \Gamma/\Nu \to \mathrm{Isom}(E^m/\Nu)$ is the homomorphism induced by the action of $\Gamma/\Nu$ on $E^m/\Nu$, 
then $\Xi = C_\sharp\eta\Rho$ 
where $\Rho: \Gamma/\Nu \to \Gamma'$ is the isomorphism defined by $\Rho(\Nu\gamma) = \gamma'$. 
Define 
$$\psi: \mathrm{Aff}(\Delta,\Mu) \to \mathrm{Iso}(\Delta, \Mu)$$
by $\psi([\eta]) = [\Gamma,\Nu]$.  We next show that $\psi$ is well-defined. 

\begin{lemma}\label{L:3} % 3
The function $\psi: \mathrm{Aff}(\Delta,\Mu) \to \mathrm{Iso}(\Delta, \Mu)$ is a well-defined surjection. 
\end{lemma}
\begin{proof}
To see that $\psi$ is well-defined, let $\eta \in \mathrm{Hom}_f(\Delta,\mathrm{Aff}(\Mu))$, 
and let $\alpha \in \mathrm{Aff}(\Mu)$ and $\zeta \in \mathrm{Aut}(\Delta)$. 
By Theorem 5 of \cite{R-T-I}, there exists $\hat C \in \mathrm{GL}(m,\realnos)$ such that $\hat C\Mu\hat C^{-1}$ is an $m$-space group 
and 
$$\hat C_\#(\Omega\alpha_\ast\eta\zeta(\Delta))\subseteq \mathrm{Out}_E(\hat C\Mu\hat C^{-1}).$$
Then $\hat C_\sharp\alpha_\ast\eta\zeta(\Delta) \subseteq \mathrm{Isom}(E^m/\hat C\Mu\hat C^{-1})$. 
Extend $\hat C\Mu\hat C^{-1}$ to a subgroup $\hat\Nu$ of $\mathrm{Isom}(E^n)$ such that the point group of $\hat\Nu$ 
acts trivially on $(E^m)^\perp$. 
By Theorem 2.9 of \cite{R-T-B}, there exists an $n$-space group $\hat\Gamma$ containing $\hat\Nu$ as a complete normal subgroup 
such that $\hat\Gamma' = \Delta$ and if $\hat\Xi: \hat\Gamma/\hat\Nu \to \mathrm{Isom}(E^m/\hat\Nu)$ is the homomorphism 
induced by the action of $\hat\Gamma/\hat\Nu$ on $E^m/\hat\Nu$, then $\hat\Xi = \hat C_\sharp\alpha_\ast\eta\zeta\hat\Rho$ 
where $\hat\Rho:\hat \Gamma/\hat\Nu \to \hat\Gamma'$ is the isomorphism defined by $\hat\Rho(\hat\Nu\gamma) = \gamma'$. 
Lift $\alpha$ to an affinity $\tilde\alpha$ of $E^m$ such that $\tilde\alpha\Mu\tilde\alpha^{-1} = \Mu$ and $\tilde\alpha_\star= \alpha$. 
Then $\tilde\alpha_\sharp = \alpha_\ast$ and we have
$$\hat C_\sharp\alpha_\ast\eta\zeta \hat\Rho \hat\Rho^{-1}\zeta^{-1} \Rho = (\hat C\tilde\alpha C^{-1})_\sharp C_\sharp \eta\Rho,$$
and so we have 
$$\hat\Xi\hat\Rho^{-1}\zeta^{-1}\Rho = (\hat C\tilde\alpha C^{-1})_\sharp \Xi.$$
By Bieberbach's theorem, there is an affinity $\tilde\zeta$ of $E^{n-m}$ such that $\tilde\zeta\Delta\tilde\zeta^{-1}$ and $\tilde\zeta_\ast = \zeta$. 
By Theorem 3.3 of \cite{R-T-B}, there is $\phi \in \mathrm{Aff}(E^n)$ such that $\phi(\Gamma,\Nu)\phi^{-1} = (\hat\Gamma,\hat\Nu)$ 
with $\ov \phi = \hat C\tilde \alpha C^{-1}$ and $\phi' = \tilde\zeta^{-1}$. 
Thus $\psi:\mathrm{Aff}(\Delta,\Mu) \to \mathrm{Iso}(\Delta, \Mu)$ is well-defined. 

To see that $\psi$ is onto, let $[\Gamma_1,\Nu_1] \in \mathrm{Iso}(\Delta,\Mu)$. 
Then we have isomorphisms $\alpha:\ov \Nu_1 \to \Mu$ and $\beta:\Gamma_1/\Nu_1\to\Delta$. 
Let $V =\mathrm{Span}(\Nu_1)$, and let $\Xi_1:\Gamma_1/\Nu_1 \to \mathrm{Isom}(V/\Nu_1)$ 
be the homomorphism induced by the action of $\Gamma_1/\Nu_1$ on $V/\Nu_1$. 
By Bieberbach's theorem, there is an affinity $\tilde\alpha: V \to E^m$ such that $\tilde\alpha\ov\Nu_1\tilde\alpha^{-1} = \Mu$ 
and $\tilde\alpha_\ast = \alpha$. 
Let $\eta = \tilde\alpha_\sharp\Xi_1\beta^{-1}$. 
Then $\eta \in \mathrm{Hom}_f(\Delta,\mathrm{Aff}(\Mu))$ by Theorem 2 of \cite{R-T-I} and Lemma \ref{L:2}. 
By Theorems 2, 3, and 5 of \cite{R-T-I}, there is a $C\in\mathrm{GL}(m,\realnos)$ such that $C\Mu C^{-1}$ is an $m$-space group and 
$$C_\sharp\eta(\Delta) \subseteq \mathrm{Isom}(E^m/C\Mu C^{-1}).$$
Extend $C\Mu C^{-1}$ to a subgroup $\Nu_2$ of $\mathrm{Isom}(E^n)$ such that the point group of $\Nu_2$ 
acts trivially on $(E^m)^\perp$. 
By Theorem 2.9 of \cite{R-T-B}, there exists an $n$-space group $\Gamma_2$ containing $\Nu_2$ as a complete normal subgroup 
such that $\Gamma_2' = \Delta$ and if $\Xi_2: \Gamma_2/\Nu_2 \to \mathrm{Isom}(E^m/\Nu_2)$ is the homomorphism 
induced by the action of $\Gamma_2/\Nu_2$ on $E^m/\Nu_2$, then $\Xi_2 = C_\sharp\eta\Rho_2$ 
where $\Rho_2: \Gamma_2/\Nu_2 \to \Gamma_2'$ is the isomorphism defined by $\Rho_2(\Nu_2\gamma) = \gamma'$. 
Then we have
$$\Xi_2\Rho_2^{-1}(\beta\Rho_1^{-1})\Rho_1 = C_\sharp\eta\beta = C_\sharp\tilde\alpha_\sharp\Xi_1 = (C\tilde\alpha)_\sharp \Xi_1.$$
By Theorem 3.3 of \cite{R-T-B}, there is $\phi \in \mathrm{Aff}(E^n)$ such that $\phi(\Gamma_1,\Nu_1)\phi^{-1} = (\Gamma_2,\Nu_2)$. 
Therefore $\psi([\eta]) = [\Gamma_1,\Nu_1]$.  
Thus $\psi$ is surjective. 
\end{proof}

\begin{theorem}\label{T:5} % 5
Let $m$ be a positive integer less than $n$.  Let $\Mu$ be an $m$-space group 
with trivial center, and let $\Delta$ be an $(n-m)$-space group. 
Then the function $\psi:\mathrm{Aff}(\Delta,\Mu) \to \mathrm{Iso}(\Delta,\Mu)$ is a bijection.
\end{theorem}
\begin{proof}  
To see that $\psi$ is injective, let $\eta_i\in\mathrm{Hom}_f(\Delta,\mathrm{Aff}(\Mu))$ for $i = 1,2$ 
such that $\psi([\eta_1]) = \psi([\eta_2])$. 
By the definition of $\psi$, there exists $C_i\in \mathrm{GL}(m,\realnos)$ such that $C_i\Mu C_i^{-1}$ is an $m$-space group 
and $(C_i)_\sharp\eta(\Delta) \subseteq \mathrm{Isom}(E^m/C_i\Mu C_i^{-1})$ for each $i=1,2$. 
Extend $C_i\Mu C_i^{-1}$ to a subgroup $\Nu_i$ of $\mathrm{Isom}(E^n)$ such that the point group of $\Nu_i$ 
acts trivially on $(E^m)^\perp$. 
By Theorem 2.9 of \cite{R-T-B}, there exists an $n$-space group $\Gamma_i$ containing $\Nu_i$ as a complete normal subgroup 
such that $\Gamma_i' = \Delta$ and if $\Xi_i: \Gamma_i\Nu_i \to \mathrm{Isom}(E^m/\Nu_i)$ is the homomorphism 
induced by the action of $\Gamma_i/\Nu_i$ on $E^m/\Nu_i$, then $\Xi_i = (C_i)_\sharp\eta_i\Rho_i$ 
where $\Rho_i: \Gamma_i/\Nu_i \to \Gamma_i'$ is the isomorphism defined by $\Rho_i(\Nu_i\gamma) = \gamma'$ for $i = 1,2$. 
Then we have 
$$[\Gamma_1,\Nu_1] = \psi([\eta_1] )= \psi([\eta_2]) = [\Gamma_2,\Nu_2].$$
Hence, there exists $\phi\in \mathrm{Aff}(E^n)$ such that $\phi(\Gamma_1,\Nu_1)\phi^{-1} = (\Gamma_2,\Nu_2)$. 
By Lemma \ref{L:2} and Theorem 3.3 of \cite{R-T-B}, we have that 
$\Xi_2\Rho_2^{-1}\phi'_\ast\Rho_1 = \ov \phi_\sharp\Xi_1$. 
Hence, we have that 
$$(C_2)_\sharp\eta_2\phi'_\ast\Rho_1 = \ov\phi_\sharp(C_1)_\sharp\eta_1\Rho_1.$$
Therefore, we have 
$$\eta_2\phi'_\ast = (C_2)_\sharp^{-1}\ov\phi_\sharp(C_1)_\sharp\eta_1 = (C_2^{-1}\ov\phi C_1)_\sharp\eta_1.$$
Hence $[\eta_1] = [\eta_2]$.  Thus $\psi$ is injective. 
By Lemma \ref{L:3}, we have that $\psi$ is surjection.  
Therefore $\psi$ is a bijection. 
\end{proof}

Let $\eta \in \mathrm{Hom}_f(\Delta, \mathrm{Aff}(\Mu))$.  
Then $\Omega\eta \in \mathrm{Hom}_f(\Delta, \mathrm{Out}(\Mu))$. 
Let $\alpha \in \mathrm{Aff}(\Mu)$ and $\beta\in \mathrm{Aut}(\Delta)$.
Then $\alpha$ lifts to $\tilde\alpha \in \mathrm{Aff}(E^m)$ such that 
$\tilde\alpha\Mu\tilde\alpha^{-1} = \Mu$ and $\tilde\alpha_\star = \alpha$. 
We have that $\tilde\alpha_\sharp = \alpha_\ast$. 
By Lemma \ref{L:2}, we have that 
$$\Omega\alpha_\ast\eta\beta = \Omega\tilde\alpha_\sharp \eta\beta 
= \tilde\alpha_\#(\Omega\eta)\beta = (\tilde\alpha_\ast)_\#(\Omega\eta)\beta.$$ 
Hence, we may define a function $\chi:\mathrm{Aff}(\Delta,\Mu) \to \mathrm{Out}(\Delta,\Mu)$ 
by $\chi([\eta]) = [\Omega\eta]$. 

%\newpage
\begin{lemma}\label{L:4}  % 4
The function $\chi:\mathrm{Aff}(\Delta,\Mu) \to \mathrm{Out}(\Delta,\Mu)$ is the composition of the function 
$\psi:\mathrm{Aff}(\Delta,\Mu) \to \mathrm{Iso}(\Delta,\Mu)$ followed by 
$\omega:\mathrm{Iso}(\Delta,\Mu) \to \mathrm{Out}(\Delta,\Mu)$.
\end{lemma}
\begin{proof}
Let $\eta \in \mathrm{Hom}_f(\Delta, \mathrm{Aff}(\Mu))$.  
By Theorem 5 of \cite{R-T-I}, there exists $C \in \mathrm{GL}(m,\realnos)$ such that $C\Mu C^{-1}$ is an $m$-space group 
and 
$C_\#(\Omega\eta(\Delta))\subseteq \mathrm{Out}_E(\hat C\Mu C^{-1}).$
Then $C_\sharp\eta(\Delta) \subseteq \mathrm{Isom}(E^m/C\Mu C^{-1})$. 
Extend $C\Mu C^{-1}$ to a subgroup $\Nu$ of $\mathrm{Isom}(E^n)$ such that the point group of $\Nu$ 
acts trivially on $(E^m)^\perp$. 
By Theorem 2.9 of \cite{R-T-B}, there exists an $n$-space group $\Gamma$ containing $\Nu$ as a complete normal subgroup 
such that $\Gamma' = \Delta$ and if $\Xi: \Gamma/\Nu \to \mathrm{Isom}(E^m/\Nu)$ is the homomorphism 
induced by the action of $\Gamma/\Nu$ on $E^m/\Nu$, then $\Xi = C_\sharp\eta\Rho$ 
where $\Rho:\Gamma/\Nu \to \Gamma'$ is the isomorphism defined by $\Rho(\Nu\gamma) = \gamma'$. 
Then $\psi([\eta]) = [\Gamma,\Nu]$. 

The homomorphism $\mathcal{O}: \Gamma/\Nu \to \mathrm{Out}_E(\Nu)$ induced by the action of $\Gamma$ on $\Nu$ by 
conjugation is given by $\mathcal{O} = \Omega C_\sharp \eta \Rho$.  Now, we have that 
$$\omega([\Gamma,\Nu]) = [C_\#^{-1}\mathcal{O}\Rho^{-1}] = [C_\#^{-1}(\Omega C_\sharp \eta\Rho)\Rho^{-1}] 
= [\Omega\eta] = \chi([\eta]).$$
Therefore $\chi = \omega\psi$. 
\end{proof}

\begin{corollary}\label{C:2} % 2
If $\Mu$ has trivial center, then the function $\omega:\mathrm{Iso}(\Delta,\Mu) \to \mathrm{Out}(\Delta,\Mu)$ is a bijection. 
\end{corollary}
\begin{proof}
We have that $\Omega: \mathrm{Aff}(\Mu) \to \mathrm{Out}(\Mu)$ is an isomorphism by Theorems 1, 2, and 3 of \cite{R-T-I}. 
Hence $\chi:\mathrm{Aff}(\Delta,\Mu) \to \mathrm{Out}(\Delta,\Mu)$ is a bijection. 
We have that $\psi:\mathrm{Aff}(\Delta,\Mu) \to \mathrm{Iso}(\Delta,\Mu)$ is a bijection by Theorem \ref{T:5}. 
By Lemma \ref{L:4}, we have that $\chi = \omega\psi$. 
Therefore $\omega$ is a bijection. 
\end{proof}

\section{Co-Seifert Geometric Fibrations Over A Circle}\label{S:7} % 7

In this section, we describe the classification of the geometric fibrations 
of compact, connected, flat $n$-orbifolds, over a circle, up to affine equivalence.  
By Theorems 5 and 10 of \cite{R-T}, this is equivalent to classifying all pairs $(\Gamma, \Nu)$, 
consisting of an $n$-space group $\Gamma$ and a normal subgroup $\Nu$ such that $\Gamma/\Nu$ is infinite cyclic, up to  isomorphism.  

Let $\Nu$ be a normal subgroup of an $n$-space group $\Gamma$ such that $\Gamma/\Nu$ is infinite cyclic. 
Let $\gamma$ be an element of $\Gamma$ such that $\gamma\Nu$ is a generator of $\Gamma/\Nu$. 
Then $\Gamma$ is an HNN extension with base $\Nu$, stable letter $\gamma$, and automorphism $\gamma_\ast$ of $\Nu$ 
defined by $\gamma_\ast(\nu)= \gamma\nu\gamma^{-1}$ for each $\nu$ in $\Nu$.
If $\mu \in \Nu$, then $\gamma\mu\Nu = \gamma\Nu$ and $(\gamma\mu)_\ast = \gamma_\ast\mu_\ast$. 
Hence, the generator $\gamma\Nu$ of $\Gamma/\Nu$ determines a 
unique element $\gamma_\ast\mathrm{Inn}(N)$ of $\mathrm{Out}_E(\Nu)$. 
The other generator $\gamma^{-1}\Nu$ of $\Gamma/\Nu$ 
determines the element $\gamma_\ast^{-1}\mathrm{Inn}(N)$ of $\mathrm{Out}_E(\Nu)$. 
Hence, the pair $(\Gamma, \Nu)$ determines the pair of inverse elements 
$\{\gamma_\ast\mathrm{Inn}(N), \gamma_\ast^{-1}\mathrm{Inn}(N)\}$ of $\mathrm{Out}_E(\Nu)$. 
As usual $(x,y)$ denotes an ordered pair whereas $\{x,y\}$ denotes an unordered pair. 

\begin{lemma}\label{L:5} % 5
Let $\Nu_i$ be a normal subgroup of an $n$-space group $\Gamma_i$ 
such that $\Gamma_i/\Nu_i$ is infinite cyclic for $i =1,2$.  
Let $\gamma_i\in \Gamma_i$ be such that $\gamma_i\Nu_i$ generates $\Gamma_i/\Nu_i$ for $i =1,2$. 
Then an isomorphism $\alpha: \Nu_1 \to \Nu_2$ extends to an isomorphism $\phi:\Gamma_1 \to \Gamma_2$ 
if and only if $\alpha(\gamma_1)_\ast\alpha^{-1} \mathrm{Inn}(\Nu_2) = (\gamma_2^{\pm 1})_\ast\mathrm{Inn}(\Nu_2)$. 
\end{lemma}
\begin{proof}
Suppose $\alpha$ extends to an isomorphism $\phi:\Gamma_1 \to \Gamma_2$. 
Then $\phi(\gamma_1)$ is an element of $\Gamma_2$ such that $\phi(\gamma_1)\Nu_2$ generates $\Gamma_2/\Nu_2$. 
Hence $\phi(\gamma_1)\Nu_2 = \gamma^{\pm 1}_2\Nu_2$. 
If $\nu \in \Nu_2$, then 
$$(\alpha(\gamma_1)_\ast\alpha^{-1})(\nu) =\alpha(\gamma_1\alpha^{-1}(\nu)\gamma_1^{-1}) 
= \phi(\gamma_1)\nu\phi(\gamma_1)^{-1}.$$
Hence $\alpha(\gamma_1)_\ast\alpha^{-1} = (\phi(\gamma_1))_\ast$, and  
so $\alpha(\gamma_1)_\ast\alpha^{-1} \mathrm{Inn}(\Nu_2) = (\gamma_2^{\pm 1})_\ast\mathrm{Inn}(\Nu_2)$. 

Conversely, suppose that 
$\alpha(\gamma_1)_\ast\alpha^{-1} \mathrm{Inn}(\Nu_2) = (\gamma_2^{\pm 1})_\ast\mathrm{Inn}(\Nu_2)$. 
By replacing $\gamma_2$ by $\gamma_2^{-1}$, if necessary, we may assume that 
$\alpha(\gamma_1)_\ast\alpha^{-1}\mathrm{Inn}(\Nu_2) = (\gamma_2)_\ast\mathrm{Inn}(\Nu_2)$. 
Then there exists $\mu \in \Nu_2$ such that $\alpha(\gamma_1)_\ast\alpha^{-1}= (\gamma_2)_\ast\mu_\ast$. 
Hence $\alpha(\gamma_1)_\ast\alpha^{-1} = (\gamma_2\mu)_\ast$. 
Define $\phi:\Gamma_1\to \Gamma_2$ by 
$$\phi(\nu\gamma_1^k) = \alpha(\nu)(\gamma_2\mu)^k$$
for each $\nu\in\Nu_1$ and $k\in\integers$. 
If $\nu, \lambda \in \Nu_1$ and $k,\ell \in \integers$, 
then we have 
\begin{eqnarray*}
\phi(\nu\gamma_1^k\lambda\gamma_1^\ell) 
& = & \phi(\nu\gamma_1^k\lambda\gamma_1^{-k}\gamma_1^{k+\ell}) \\
& = & \phi(\nu(\gamma_1)_\ast^k(\lambda)\gamma_1^{k+\ell}) \\
& = & \alpha(\nu(\gamma_1)_\ast^k(\lambda))(\gamma_2\mu)^{k+\ell} \\
& = & \alpha(\nu)\alpha(\gamma_1)_\ast^k\alpha^{-1}\alpha(\lambda)(\gamma_2\mu)^{k+\ell} \\
& = & \alpha(\nu)(\alpha(\gamma_1)_\ast\alpha^{-1})^k\alpha(\lambda)(\gamma_2\mu)^{k+\ell} \\
& = & \alpha(\nu)(\gamma_2\mu)_\ast^k\alpha(\lambda)(\gamma_2\mu)^{k+\ell} \\
& = & \alpha(\nu)(\gamma_2\mu)^k\alpha(\lambda)(\gamma_2\mu)^{-k}(\gamma_2\mu)^{k+\ell} \\
& = & \alpha(\nu)(\gamma_2\mu)^k\alpha(\lambda)(\gamma_2\mu)^{\ell} 
\ \ = \ \ \phi(\nu\gamma_1^k)\phi(\lambda\gamma_1^\ell).
\end{eqnarray*}
Hence $\phi$ is a homomorphism;  
moreover $\phi$ is an isomorphism, since $\phi$ restricts to $\alpha$ on $\Nu_1$ and 
$\phi(\gamma_1) = \gamma_2\mu$. 
\end{proof}

\begin{lemma}\label{L:6} % 6
Let $\Mu$ be an $(n-1)$-space group, and let $\Delta$ be an infinite cyclic 1-space group 
with generator $\delta$.  
The set $\mathrm{Out}(\Delta,\Mu)$ is in one-to-one correspondence with the set 
of conjugacy classes of pairs of inverse elements of $\mathrm{Out}(\Mu)$ of finite order. 
The element $[\eta]$ of $\mathrm{Out}(\Delta,\Mu)$ corresponds to the conjugacy class of $\{\eta(\delta), \eta(\delta^{-1})\}$. 
\end{lemma}
\begin{proof}
The the set $\mathrm{Hom}_f(\Delta,\mathrm{Out}(\Mu))$ is in one-to-one correspondence 
with the set of elements of $\mathrm{Out}(\Mu)$ of finite order via the mapping $\eta\mapsto \eta(\delta)$, 
since $\Delta$ is the free group generated by $\delta$. 
The infinite cyclic group $\Delta$ has a unique automorphism that maps $\delta$ to $\delta^{-1}$ 
and this automorphism represents the generator of the group $\mathrm{Aut}(\Delta) = \mathrm{Out}(\Delta)$ 
of order 2. 
Therefore, the set $\mathrm{Out}(\Delta,\Mu)$ is in one-to-one correspondence with the set 
of conjugacy classes of pairs of inverse elements of $\mathrm{Out}(\Mu)$ of finite order via 
the mapping $[\eta] \mapsto [\{\eta(\delta), \eta(\delta^{-1}\}]$. 
\end{proof}

\begin{theorem}\label{T:6} % 6
If $\Delta$ is an infinite cyclic $1$-space group and $\Mu$ is an $(n-1)$-space group, 
then the function $\omega: \mathrm{Iso}(\Delta,\Mu) \to \mathrm{Out}(\Delta,\Mu)$ is a bijection.
\end{theorem}
\begin{proof} Any homomorphism $\eta:\Delta \to \mathrm{Out}(\Mu)$ lifts to a homomorphism 
$\tilde\eta:\Delta\to\mathrm{Aff}(\Mu)$ such that $\Omega\tilde\eta = \eta$, 
since $\Omega: \mathrm{Aff}(\Mu) \to \mathrm{Out}(\Mu)$ is onto by Theorem 3 of \cite{R-T-I}. 
Hence, the function $\chi:\mathrm{Aff}(\Delta,\Mu) \to \mathrm{Out}(\Delta,\Mu)$ is surjective. 
Therefore $\omega$ is surjective, since $\chi = \omega\psi$ by Lemma \ref{L:4} 

Let $\Nu_i$ be a normal subgroup of an $n$-space group $\Gamma_i$ 
with $\Gamma_i/\Nu_i$ infinite cyclic for $i = 1,2$ such that 
$\omega([\Gamma_1,\Nu_1]) = \omega([\Gamma_2,\Nu_2])$.  
Let $\alpha_i: \Nu_i \to \Mu$ be an isomorphism for $i=1,2$.  
Let $\gamma_i \in \Gamma_i$ such that $\Nu_i\gamma_i$ generates $\Gamma_i/\Nu_i$ for each $i=1,2$, 
and let $\delta$ be a generator of $\Delta$. 
Define an isomorphism $\beta_i:\Delta\to \Gamma_i/\Nu_i$ by $\beta_i(\delta) = \gamma_i$ for $i =1,2$. 
Let $\mathcal{O}_i: \Gamma_i/\Nu_i \to \mathrm{Out}_E(\Nu_i)$ be the homomorphism induced by the action 
of $\Gamma_i$ on $\Nu_i$ by conjugation for $i=1,2$. 
Then in $\mathrm{Out}(\Delta,\Mu)$, we have 
$$[(\alpha_1)_\#\mathcal{O}_1\beta_1] = [(\alpha_2)_\#\mathcal{O}_2\beta_2].$$
Hence, by Lemma \ref{L:6}, there is an automorphism $\zeta$ of $\Mu$ such that 
$$\{\zeta\alpha_1(\gamma_1^{\pm 1})_\ast\alpha_1^{-1}\zeta^{-1}\,\mathrm{Inn}(\Mu)\} = 
\{\alpha_2(\gamma_2^{\pm 1})_\ast\alpha_2^{-1}\mathrm{Inn}(\Mu)\}.$$
After applying the isomorphism $(\alpha_2^{-1})_\#: \mathrm{Out}(\Mu) \to \mathrm{Out}(\Nu_2)$, we have that 
$$\{\alpha_2^{-1}\zeta\alpha_1(\gamma_1^{\pm 1})_\ast\alpha_1^{-1}\zeta^{-1}\alpha_2\mathrm{Inn}(\Nu_2)\} = 
\{(\gamma_2^{\pm 1})_\ast\mathrm{Inn}(\Nu_2)\}.$$
Hence, the isomorphism $\alpha_2^{-1}\zeta\alpha_1: \Nu_1 \to \Nu_2$ extends to an isomorphism 
$\phi:(\Gamma_1, \Nu_1) \to (\Gamma_2,\Nu_2)$ by Lemma \ref{L:5}.  
Therefore $[\Gamma_1,\Nu_1]= [\Gamma_2,\Nu_2]$, and so $\omega$ is injective. 
\end{proof}

\begin{theorem}\label{T:7} % 7
Let $\Mu$ be an $(n-1)$-space group, and let $\Delta$ be an infinite cyclic 1-space group.  
The set $\mathrm{Iso}(\Delta,\Mu)$ is in one-to-one correspondence with the 
set of conjugacy classes of pairs of inverse elements of $\mathrm{Out}(\Mu)$ of finite order.  
If $[\Gamma,\Nu] \in \mathrm{Iso}(\Delta,\Mu)$ and 
$\alpha: \Nu \to \Mu$ is an isomorphism and $\gamma$ is an element of $\Gamma$ such that $\Nu\gamma$ generates 
$\Gamma/\Nu$, then $[\Gamma, \Nu]$ corresponds to the conjugacy class 
of the pair of inverse elements $\{\alpha\gamma_\ast^{\pm 1}\alpha^{-1}\mathrm{Inn}(\Mu)\}$ of $\mathrm{Out}(\Mu)$.   
\end{theorem}
\begin{proof}
The function $\omega: \mathrm{Iso}(\Delta,\Mu) \to \mathrm{Out}(\Delta,\Mu)$ is a bijection by Theorem \ref{T:6}. 
Let $[\Gamma,\Nu]\in\mathrm{Iso}(\Delta,\Mu)$, 
and let $\gamma$ be an element of $\Gamma$ such that $\Nu\gamma$ generates $\Gamma/\Nu$. 
Let  $\mathcal{O}: \Gamma/\Nu \to \mathrm{Out}_E(\Nu)$ 
be the homomorphism induced by the action of $\Gamma$ on $\Nu$ by conjugation. 
Let $\alpha: \Nu \to \Mu$ and $\beta:\Delta \to \Gamma/\Nu$ be isomorphisms.  
Then $\omega([\Gamma,\Nu]) = [\alpha_\#\mathcal{O}\beta]$, 
and $ [\alpha_\#\mathcal{O}\beta]$ corresponds to the conjugacy class of the pair of inverse elements 
\begin{eqnarray*}
\{\alpha_\#\mathcal{O}\beta(\beta^{-1}(\Nu\gamma^{\pm 1}))\}  
& = & \{\alpha_\#\mathcal{O}(\Nu\gamma^{\pm 1})\}  \\ 
& = & \{\alpha_\#(\gamma_\ast^{\pm 1}\mathrm{Inn}(\Nu))\}  
\ \ = \ \ \{\alpha\gamma_\ast^{\pm 1}\alpha^{-1}\mathrm{Inn}(\Mu)\}
\end{eqnarray*}
of $\mathrm{Out}(\Mu)$ by Lemma \ref{L:6}. 
\end{proof}

\section{Co-Seifert Geometric Fibrations Over A Closed Interval}\label{S:8} % 8

In this section, we describe the classification of the geometric fibrations 
of compact, connected, flat $n$-orbifolds, over a closed interval, up to affine equivalence.  
By Theorems 5 and 10 of \cite{R-T} this is equivalent to classifying all pairs $(\Gamma, \Nu)$, 
consisting of an $n$-space group $\Gamma$ and a normal subgroup $\Nu$ such that $\Gamma/\Nu$ is infinite dihedral, 
up to  isomorphism.  

Let $\Delta$ be an infinite dihedral group.  
A set of {\it Coxeter generators} of $\Delta$ is a pair of elements of $\Delta$ of order 2 that generate $\Delta$. 
Any two sets of Coxeter generators of $\Delta$ are conjugate in $\Delta$. 

\begin{lemma}\label{L:7} % 7
Let $\Mu$ be an $(n-1)$-space group, and let $\Delta$ be an infinite dihedral 1-space group 
with Coxeter generators $\delta_1$ and $\delta_2$.  
The set $\mathrm{Out}(\Delta,\Mu)$ is in one-to-one correspondence with the set 
of conjugacy classes of pairs of elements of $\mathrm{Out}(\Mu)$ of order 1 or  2 whose product has finite order. 
The element $[\eta]$ of $\mathrm{Out}(\Delta,\Mu)$ corresponds to the conjugacy class of 
$\{\eta(\delta_1),\eta(\delta_2)\}$. 
\end{lemma}
\begin{proof}
The set $\mathrm{Hom}_f(\Delta,\mathrm{Out}(\Mu))$ is in one-to-one correspondence 
with the set of ordered pairs of elements of $\mathrm{Out}(\Mu)$, of order 1 or 2  whose product has finite order, 
via the mapping $\eta\mapsto (\eta(\delta_1),\eta(\delta_2))$, 
since $\Delta$ is the free product of the cyclic groups of order two generated by $\delta_1$ and $\delta_2$. 
The infinite dihedral group $\Delta$ has a unique automorphism that transposes $\delta_1$ and $\delta_2$,  
and this automorphism represents the generator of the group $\mathrm{Out}(\Delta)$ of order 2. 
Therefore, the set $\mathrm{Out}(\Delta,\Mu)$ is in one-to-one correspondence with the set 
of conjugacy classes of unordered pairs of elements of $\mathrm{Out}(\Mu)$, of order 1 or  2 whose product has 
of finite order, via the mapping $[\eta] \mapsto [\{\eta(\delta_1), \eta(\delta_2\}]$. 
\end{proof}

\begin{theorem}\label{T:8} % 8
Let $\Mu$ be an $(n-1)$-space group with trivial center, 
and let $\Delta$ be an infinite dihedral 1-space group.  
The set $\mathrm{Iso}(\Delta,\Mu)$ is in one-to-one correspondence 
with the set of conjugacy classes of pairs of elements of $\mathrm{Out}(\Mu)$ of order 1 or  2 
whose product has finite order. 
If $[\Gamma,\Nu] \in \mathrm{Iso}(\Delta,\Mu)$ and $\alpha: \Nu \to \Mu$ is an isomorphism,  
and $\gamma_1, \gamma_2$ are elements of $\Gamma$ 
such that $\{\Nu\gamma_1, \Nu\gamma_2\}$ is a set of Coxeter generators of $\Gamma/\Nu$, 
then $[\Gamma, \Nu]$ corresponds to the conjugacy class of the pair of elements 
$\{\alpha(\gamma_1)_\ast\alpha^{-1} \mathrm{Inn}(\Mu), \alpha(\gamma_2)_\ast\alpha^{-1}\mathrm{Inn}(\Mu)\}$ 
of $\mathrm{Out}(\Mu)$.   
\end{theorem}
\begin{proof}
The function $\omega: \mathrm{Iso}(\Delta,\Mu) \to \mathrm{Out}(\Delta,\Mu)$ is a bijection by Corollary \ref{C:2}. 
Let $[\Gamma,\Nu]\in\mathrm{Iso}(\Delta,\Mu)$, 
and let $\gamma_1,\gamma_2$ be elements of $\Gamma$ such that 
$\{\Nu\gamma_1, \Nu\gamma_2\}$ is a set of Coxeter generators of $\Gamma/\Nu$. 
Let  $\mathcal{O}: \Gamma/\Nu \to \mathrm{Out}_E(\Nu)$ 
be the homomorphism induced by the action of $\Gamma$ on $\Nu$ by conjugation. 
Let $\alpha: \Nu \to \Mu$ and $\beta:\Delta \to \Gamma/\Nu$ be isomorphisms.  
Then $\omega([\Gamma,\Nu]) = [\alpha_\#\mathcal{O}\beta]$, 
and $ [\alpha_\#\mathcal{O}\beta]$ corresponds to the conjugacy class of the pair of elements 
\begin{eqnarray*}
\{\alpha_\#\mathcal{O}\beta\beta^{-1}(\Nu\gamma_1), \alpha_\#\mathcal{O}\beta\beta^{-1}(\Nu\gamma_2)\} 
& =  & \{\alpha_\#\mathcal{O}(\Nu\gamma_1), \alpha_\#\mathcal{O}(\Nu\gamma_2)\}  \\ 
& =  & \{\alpha_\#(\gamma_1)_\ast\mathrm{Inn}(\Nu)), \alpha_\#(\gamma_2)_\ast\mathrm{Inn}(\Nu))\}  \\
& =  & \{\alpha(\gamma_1)_\ast\alpha^{-1}\mathrm{Inn}(\Mu), \alpha(\gamma_2)_\ast\alpha^{-1}\mathrm{Inn}(\Mu)\}
\end{eqnarray*}
of $\mathrm{Out}(\Mu)$ by Lemma \ref{L:7}. 
\end{proof}

Let $\Mu$ be an $(n-1)$-space group, and let $\Delta$ be an infinite dihedral 1-space group. 
When $\Mu$ has nontrivial center, the set $\mathrm{Iso}(\Delta,\Mu)$ is best understood by describing the fibers 
of  the surjection $\psi: \mathrm{Aff}(\Delta,\Mu) \to \mathrm{Iso}(\Delta,\Mu)$.

\begin{lemma}\label{L:8} % 8
Let $\Mu$ be an $(n-1)$-space group, and let $\Delta$ be an infinite dihedral 1-space group 
with Coxeter generators $\delta_1$ and $\delta_2$.  
The set $\mathrm{Aff}(\Delta,\Mu)$ is in one-to-one correspondence with the set 
of conjugacy classes of pairs of elements of $\mathrm{Aff}(\Mu)$ of order 1 or  2 whose product has image 
of finite order under the epimorphism $\Omega: \mathrm{Aff}(\Mu) \to \mathrm{Out}(\Mu)$. 
The element $[\eta]$ of $\mathrm{Aff}(\Delta,\Mu)$ corresponds to the conjugacy class of 
$\{\eta(\delta_1),\eta(\delta_2)\}$. 
\end{lemma}
\begin{proof}
Via the mapping $\eta\mapsto (\eta(\delta_1),\eta(\delta_2))$, 
the set $\mathrm{Hom}_f(\Delta,\mathrm{Aff}(\Mu))$ is in one-to-one correspondence 
with the set of ordered pairs of elements of $\mathrm{Aff}(\Mu)$ of order 1 or 2  whose product has image 
of finite order under the epimorphism $\Omega: \mathrm{Aff}(\Mu) \to \mathrm{Out}(\Mu)$, 
since $\Delta$ is the free product of the cyclic groups of order two generated by $\delta_1$ and $\delta_2$.  
The infinite dihedral group $\Delta$ has a unique automorphism that transposes $\delta_1$ and $\delta_2$,  
and this automorphism represents the generator of the group $\mathrm{Out}(\Delta)$ of order 2. 
Therefore, via the mapping $[\eta] \mapsto [\{\eta(\delta_1), \eta(\delta_2\}]$,  
the set $\mathrm{Aff}(\Delta,\Mu)$ is in one-to-one correspondence with the set 
of conjugacy classes of unordered pairs of elements of $\mathrm{Aff}(\Mu)$ of order 1 or  2 whose product has image 
of finite order under the epimorphism $\Omega: \mathrm{Aff}(\Mu) \to \mathrm{Out}(\Mu)$. 
\end{proof}

We denote the identity maps of $E^n, E^{n-1}, E^1$ by $I, \ov I, I'$, respectively, 
and we identify $E^1$ with $(E^{n-1})^\perp$ in $E^n$. 
To simplify matters, we assume that $\Delta$ is the {\it standard} 
infinite dihedral 1-space group with Coxeter generators the reflections $\delta_1=-I'$ and $\delta_2= 1-I'$ of $E^1$.  
The next theorem gives an algebraic description of 
the fibers of the surjection  $\psi: \mathrm{Aff}(\Delta,\Mu) \to \mathrm{Iso}(\Delta,\Mu)$. 

%\newpage
\begin{theorem}\label{T:9}  % 9
Let let $\Mu$ be an $(n-1)$-space group, and let 
$\Delta$ be the infinite dihedral $1$-space group generated by $\delta_1=-I'$ and $\delta_2= 1-I'$.  
Let $\eta_1,\eta_2 \in \mathrm{Hom}_f(\Delta,\mathrm{Aff}(\Mu))$, 
and let $\eta_1(\delta_i) = (a_i+A_i)_\star$, with $a_i+A_i \in \mathrm{Aff}(E^{n-1})$ 
that normalizes $\Mu$  for $i = 1,2$. 
Let $E_i$ be the $(-1)$-eigenspace of the restriction of $A_i$ to $\mathrm{Span}(Z(\Mu))$ for $i = 1,2$. 
Then the surjection $\psi: \mathrm{Aff}(\Delta,\Mu) \to \mathrm{Iso}(\Delta,\Mu)$ 
has the property that $\psi([\eta_1]) = \psi([\eta_2])$ if and only if 
there is a vector $v \in E_1\cap E_2$
such that $\{\eta_1(\delta_1), (v+\ov I)_\star \eta_1(\delta_2)\}$ is conjugate to $\{\eta_2(\delta_1),\eta_2(\delta_2)\}$ 
by an element of $\mathrm{Aff}(\Mu)$. 
\end{theorem}
\begin{proof}
By Theorem 5 of \cite{R-T-I}, there exists $C_i\in \mathrm{GL}(n-1,\realnos)$ such that $C_i\Mu C_i^{-1}$ is an $(n-1)$-space group 
and $(C_i)_\sharp\eta_i(\Delta) \subseteq \mathrm{Isom}(E^{n-1}/C_i\Mu C_i^{-1})$ for $i=1,2$. 
Extend  $C_i\Mu C_i^{-1}$ to a subgroup $\Nu_i$ of $\mathrm{Isom}(E^n)$ such that 
the point group of $\Nu_i$ acts trivially on $(E^{n-1})^\perp$ for $i = 1,2$. 
By Theorem 2.9 of \cite{R-T-B}, there exists an $n$-space group $\Gamma_i$ containing $\Nu_i$ as a complete normal subgroup 
such that $\Gamma_i' = \Delta$ and if $\Xi_i: \Gamma_i/\Nu_i \to \mathrm{Isom}(E^{n-1}/\Nu_i)$ is the homomorphism 
induced by the action of $\Gamma_i/\Nu_i$ on $E^{n-1}/\Nu_i$, then $\Xi_i = (C_i)_\sharp\eta_i\Rho_i$ 
where $\Rho_i:\Gamma_i/\Nu_i \to \Gamma_i'$ is the isomorphism defined by $\Rho_i(\Nu_i\gamma) = \gamma'$. 
Then $\psi([\eta_i]) = [\Gamma_i,\Nu_i]$ for $i=1,2$. 

Suppose $\psi([\eta_1]) = \psi([\eta_2])$.  
Then $[\Gamma_1,\Nu_1] = [\Gamma_2,\Nu_2]$. 
By Bieberbach's theorem, there exists an affinity $\phi = c + C$ of $E^n$ such that 
$\phi(\Gamma_1,\Nu_1)\phi^{-1} = (\Gamma_2,\Nu_2)$. 
Then $C(E^{n-1}) = E^{n-1}$, since $\phi(\Nu_1) = \Nu_2$.  
Let $\overline{C}:E^{n-1} \to E^{n-1}$ be the linear automorphism obtained by restricting $C$. 
Let $\overline{C'}:(E^{n-1})^\perp \to E^{n-1}$ and $C': (E^{n-1})^\perp \to (E^{n-1})^\perp$ 
be the linear transformations obtained by restricting $C$ to $(E^{n-1})^\perp$ followed 
by orthogonal projections to $E^{n-1}$ and $(E^{n-1})^\perp$ respectively. 
Write $c = \overline{c}+c'$ with $\overline{c}\in E^{n-1}$ and $c' \in (E^{n-1})^\perp$. 
Let $\overline{\phi}: E^{n-1} \to E^{n-1}$ and $\phi':(E^{n-1})^\perp \to (E^{n-1})^\perp$ 
be the affine transformations defined by $\overline{\phi} = \overline{c}+\overline{C}$ 
and $\phi' = c'+ C'$. 

By Theorem 3.3 in \cite{R-T-B}, we have that 
$$\Xi_2\Rho_2^{-1}(\phi')_\ast\Rho_1 = (\ov{C'}p_1)_\star(\ov\phi)_\sharp\Xi_1.$$
Hence, we have 
$$(C_2)_\sharp\eta_2(\phi')_\ast= ((\ov{C'}p_1)_\star\Rho_1^{-1})(\ov\phi)_\sharp(C_1)_\sharp\eta_1,$$
Now $(\phi')_\star(\{\delta_1,\delta_2\} )$ is a set of Coxeter generators of $\Delta$, 
and hence there exists $\delta \in \Delta$ such that $(\phi')_\star(\{\delta_1,\delta_2\} )= \delta\{\delta_1,\delta_2\}\delta^{-1}$, 
and say that $(\phi')_\star(\delta_1) = \delta\delta_j\delta^{-1}$ and $(\phi')_\star(\delta_2) = \delta\delta_k\delta^{-1}$.  
Upon evaluating at $\delta_1$ and $\delta_2$, we have 
$$(C_2)_\star\eta_2(\delta\delta_j\delta^{-1})(C_2)_\star^{-1} = (\ov\phi C_1)_\star \eta_1(\delta_1)(\ov\phi C_1)_\star^{-1}, $$
and 
$$(C_2)_\star\eta_2(\delta\delta_k\delta^{-1})(C_2)_\star^{-1} = 
(\ov{C'}(e_n) + \ov I)_\star(\ov\phi C_1)_\star \eta_1(\delta_2)(\ov\phi C_1)_\star^{-1}. $$
Therefore 
$$\eta_2(\delta_j) =
\eta_2(\delta)^{-1}(C_2^{-1}\ov\phi C_1)_\star \eta_1(\delta_1)(C_2^{-1}\ov\phi C_1)_\star^{-1}\eta_2(\delta), $$
and 
$$\eta_2(\delta_k) =
\eta_2(\delta)^{-1}(C_2^{-1}\ov\phi C_1)_\star(\ov \phi C_1)_\star^{-1}(\ov{C'}(e_n)+\ov I)_\star(\ov\phi C_1)_\star \eta_1(\delta_2)(C_2^{-1}\ov\phi C_1)_\star^{-1}\eta_2(\delta).$$
Now, we have that 
$$(\ov \phi C_1)_\star^{-1}(\ov{C'}(e_n)+\ov I)_\star(\ov\phi C_1)_\star  = (C_1^{-1}\ov C^{-1}\ov{C'}(e_n)+\ov I)_\star.$$
By Lemma 3.1 of \cite{R-T-B}, we have that 
$\ov{C'}(e_n) \in \mathrm{Span}(Z(\Nu_2))$. 
Hence, we have that $\ov{C}^{-1}\ov{C'}(e_n) \in \mathrm{Span}(Z(\Nu_1))$, 
and so $C_1^{-1}\ov{C}^{-1}\ov{C'}(e_n) \in \mathrm{Span}(Z(\Mu))$.

Next, we have that 
$$(C_1)_\sharp\eta_1(\delta_i) = (C_1a_i + C_1A_iC_1^{-1})_\star$$
for $i =1,2$. 
Let $\hat \delta_i = b_i + B_i \in \Gamma_1$ be formed from $C_1a_i+ C_1A_iC_1^{-1}$ and $\delta_i$, 
as in the beginning of the proof of Theorem 2.9 of \cite{R-T-B}, 
so that $\hat\delta_i' = \delta_i$ for $i =1,2$.  
Then $\ov B_i = C_1A_iC_1^{-1}$ and $B_i' = -I'$ for $i = 1,2$, and $b_1' = 0$ and $b_2' = e_n$. 
By Lemma 3.1 of \cite{R-T-B}, we have that 
$\ov{C'}B_i' = \ov C \ov B_i \ov C^{-1} \ov{C'}$. 
Hence $\ov{C'}(e_n)$ is in the $(-1)$-eigenspace of $\ov C \ov B_i \ov C^{-1}$. 
Therefore $\ov C^{-1}\ov{C'}(e_n)$ is in the $(-1)$-eigenspace of $\ov B_i$, 
and so $C_1^{-1}\ov C^{-1}\ov{C'}(e_n)$ is in the $(-1)$-eigenspace of $A_i$ for $i = 1,2$.  
Now let $v = C_1^{-1}\ov C^{-1}\ov{C'}(e_n)$.  Then $v \in E_1\cap E_2$, and we have that 
$\{\eta_1(\delta_1), (v+\ov I)_\star \eta_1(\delta_2)\}$ 
is conjugate to $\{\eta_2(\delta_1),\eta_2(\delta_2)\}$ by the element $\eta_2(\delta)^{-1}(C_2^{-1}\ov\phi C_1)_\star$ 
of $\mathrm{Aff}(\Mu)$. 

Conversely, suppose there is a vector $v \in E_1\cap E_2$
such that $\{\eta_1(\delta_1), (v+\ov I)_\star \eta_1(\delta_2)\}$ is conjugate to $\{\eta_2(\delta_1),\eta_2(\delta_2)\}$ 
by an element of $\mathrm{Aff}(\Mu)$. 
Let $\xi\in \mathrm{Aff}(E^{n-1})$ such that $\xi\Mu\xi^{-1} = \Mu$ and 
$\eta_2(\delta_j) = \xi_\star \eta_1(\delta_1)\xi_\star^{-1}$ and $\eta_2(\delta_k) =
 \xi_\star(v+\ov I)_\star\eta_1(\delta_2)\xi_\star^{-1}$ with $\{j,k\} = \{1,2\}$. 
Define $\alpha \in \mathrm{Aff}(E^{n-1})$ by $\alpha = C_2\xi C_1^{-1}$.  
Then $\alpha C_1\Mu C_1^{-1}\alpha^{-1} = C_2\Mu C_2^{-1}$, and so $\alpha\ov\Nu_1\alpha^{-1} = \ov \Nu_2$. 
Let $\beta \in\mathrm{Aff}((E^{n-1})^\perp)$ be either the identity map $I'$ if $j = 1$ 
or the reflection $1/2-I'$ if $k = 1$. 
Then $\beta_\ast$ is the automorphism of $\Delta$ that maps $(\delta_1,\delta_2)$ to $(\delta_j, \delta_k)$. 

Write $\xi = a + A$ with $a \in E^{n-1}$ and $A \in \mathrm{GL}(n-1,\realnos)$. 
Define a linear transformation $D: (E^{n-1})^\perp \to \mathrm{Span}(Z(\Nu_2))$ 
by $D(e_n) = C_2 A v$. 
Let $\phi = c+C \in \mathrm{Aff}(E^n)$ be such that $\ov \phi = \alpha$, and $\phi' = \beta$, and $\ov{C'} = D$. 
Then $\ov C = C_2AC_1^{-1}$.  
For $\hat\delta_i=b_i+B_i$, we have that $\ov B_i = C_1A_iC_1^{-1}$ for $i = 1,2$. 
Now $v$ is in the $(-1)$-eigenspace of $A_i$ for $i= 1,2$, 
and so $Av$ is in the $(-1)$-eigenspace of $AA_iA^{-1}$ for $i= 1,2$. 
Hence $C_2Av$ is in the $(-1)$-eigenspace of $C_2AA_iA^{-1}C_2^{-1}$ for $i= 1,2$.
Therefore $C_2Av$ is in the $(-1)$-eigenspace of $\ov C\ov B_i\ov C^{-1}$ for $i= 1,2$. 
Hence $DB_i' = \ov C\ov B_i\ov C^{-1}D$ for $i = 1,2$. 
As $\Nu_1\hat\delta_1$ and $\Nu_1\hat\delta_2$ generate $\Gamma_1/\Nu_1$, 
we have that if $b+B \in \Gamma_1$, then $DB' = \ov C \ov B\ov C^{-1}D$. 

Observer that 
\begin{eqnarray*}
\Xi_2\Rho_2^{-1}\beta_\ast\Rho_1(\Nu_1\hat\delta_1) 
& = & \Xi_2\Rho_2^{-1}\beta_\ast(\delta_1) \\
& = & (C_2)_\sharp \eta_2(\delta_j) \\ 
& = & (C_2)_\sharp \xi_\star \eta_1(\delta_1)\xi_\star^{-1} \\
& = & (C_2)_\sharp\xi_\sharp (C_1)_\sharp^{-1} (C_1)_\sharp\eta_1\Rho_1\Rho_1^{-1}(\delta_1) \\
& = & \alpha_\sharp\Xi_1(\Nu_1\hat\delta_1) \ \
 = \ \  (Dp_1)_\star(\Nu_1\hat\delta_1)\alpha_\sharp\Xi_1(\Nu_1\hat\delta_1),  
\end{eqnarray*}
and
\begin{eqnarray*}
\Xi_2\Rho_2^{-1}\beta_\ast\Rho_1(\Nu_1\hat\delta_2) 
& = & \Xi_2\Rho_2^{-1}\beta_\ast(\delta_2) \\
& = & (C_2)_\sharp \eta_2(\delta_k) \\ 
& = & (C_2)_\sharp \xi_\star(v+\ov I)_\star  \eta_1(\delta_2)\xi_\star^{-1} \\
& = & (C_2)_\star \xi_\star(v+\ov I)_\star \xi_\star^{-1}(C_2)_\star^{-1}(C_2)_\star\xi_\star\eta_1(\delta_2)\xi_\star^{-1}(C_2)_\star^{-1} \\
& = & (C_2 A v+\ov I)_\star (C_2)_\star\xi_\star(C_1)_\star^{-1} (C_1)_\star\eta_1(\delta_2)\xi_\star^{-1}(C_2)_\star^{-1} \\
& = & (De_n+\ov I)_\star\alpha_\star (C_1)_\star \eta_1(\delta_2)(C_1)^{-1}_\star \alpha_\star^{-1} \\
& = &  (Db_2'+\ov I)_\star\alpha_\sharp (C_1)_\sharp\eta_1\Rho_1\Rho_1^{-1}(\delta_2) \\
& = & (Dp_1)_\star(\Nu_1\hat\delta_2)\alpha_\sharp\Xi_1(\Nu_1\hat\delta_2).   
\end{eqnarray*}

Let $\mathcal{K}_2$ be the connected component of the identity of $\mathrm{Isom}(E^{n-1}/\Nu_2)$. 
Then $\mathcal{K}_2$ is the kernel of the epimorphism $\Omega: \mathrm{Isom}(E^{n-1}/\Nu_2) \to \mathrm{Out}_E(\Nu_2)$ 
by Theorem 2 of \cite{R-T-I}.  Upon applying $\Omega$, we have that
$$\Omega\Xi_2\Rho_2^{-1}\beta_\ast\Rho_1(\Nu_1\hat\delta_1) = \Omega\alpha_\sharp\Xi_1(\Nu_1\hat\delta_1),$$
and by Theorem 1 of \cite{R-T-I} that 
$$\Omega\Xi_2\Rho_2^{-1}\beta_\ast\Rho_1(\Nu_1\hat\delta_2) = \Omega\alpha_\sharp\Xi_1(\Nu_1\hat\delta_2).$$
Hence $\Omega\Xi_2\Rho_2^{-1}\beta_\ast\Rho_1 = \Omega\alpha_\sharp\Xi_1$, 
since $\Nu_1\hat\delta_1$ and $\Nu_1\hat\delta_2$ generate $\Gamma_1/\Nu_1$. 
Therefore, the ratio of homomorphisms $(\Xi_2\Rho_2^{-1}\beta_\ast\Rho_1)(\alpha_\sharp\Xi_1)^{-1}$ 
maps to the abelian group $\mathcal{K}_2$. 
Hence $(\Xi_2\Rho_2^{-1}\beta_\ast\Rho_1)(\alpha_\sharp\Xi_1)^{-1}: \Gamma_1/\Nu_1 \to \mathcal{K}_2$ 
is a crossed homomorphism with $\Nu_1(b+B)$ acting on $\mathcal{K}_2$ by conjugation by 
\begin{eqnarray*}
\Xi_2\Rho_2^{-1}\beta_\ast\Rho_1(\Nu_1(b+B)) 
& = & \Xi_2\Rho_2^{-1}\beta_\ast(b'+B') \\
& = & \Xi_2\Rho_2^{-1}((c'+C')(b'+B')(c'+C')^{-1})\\
& = & \Xi_2(\Nu_2(c+C)(b+B)(c+C)^{-1}) \\
& = & ((\ov c+\ov C)(\ov b+\ov B)(\ov c+\ov C)^{-1})_\star, 
\end{eqnarray*} 
that is, if $u \in \mathrm{Span}(Z(\Nu_2))$, 
then 
$$(\Nu_1(b+B))(u+I)_\star = (\ov C\ov B\ov C^{-1}u + I)_\star.$$
The mapping  $(Dp_1)_\star: \Gamma_1/\Nu_1 \to \mathcal{K}_2$ is a crossed homomorphism 
with respect to the same action of $\Gamma_1/\Nu_1$ on $\mathcal{K}_2$ 
by Theorem 3.3 of \cite{R-T-B}. 
Hence, the crossed homomorphisms $(\Xi_2\Rho_2^{-1}\beta_\ast\Rho_1)(\alpha_\sharp\Xi_1)^{-1}$ 
and $(Dp_1)_\star$ are equal, since they agree on the generators 
$\Nu_1\hat\delta_1$ and $\Nu_1\delta_2$ of $\Gamma_1/\Nu_1$. 
Therefore 
$$\Xi_2\Rho_2^{-1}\beta_\ast\Rho_1= (Dp_1)_\star\alpha_\sharp\Xi_1.$$
Hence $\phi(\Gamma_1,\Nu_1)\phi^{-1} = (\Gamma_2,\Nu_2)$ by Theorem 3.3 of \cite{R-T-B}. 
Therefore, we have that 
$$\psi([\eta_1]) = [\Gamma_1,\Nu_1] = [\Gamma_2,\Nu_2] = \psi([\eta_2]).$$

\vspace{-.2in}
\end{proof}

\begin{theorem}\label{T:10} % 10
Let let $\Mu$ be an $(n-1)$-space group, and let 
$\Delta$ be the infinite dihedral $1$-space group generated by $\delta_1=-I'$ and $\delta_2= 1-I'$.  
Let $[\Gamma,\Nu] \in \mathrm{Iso}(\Delta,\Mu)$,  let $V =  \mathrm{Span}(\Nu)$, 
and let $\alpha: V \to E^{n-1}$ be an affinity such that $\alpha\ov \Nu \alpha^{-1} =  \Mu$. 
Let $\{\Nu\gamma_1,\Nu\gamma_2\}$ be a set of Coxeter generators of $\Gamma/\Nu$, and 
let $\beta:\Delta \to \Gamma/\Nu$ be the isomorphism defined by $\beta(\delta_i) = \Nu\gamma_i$ for $i = 1, 2$. 
Let $\Xi: \Gamma/\Nu \to \mathrm{Isom}(V/\Nu)$ be the homomorphism induced by the action of $\Gamma/\Nu$ 
on $V/\Nu$. 
Write $\gamma_i = b_i + B_i$ with $b_i \in E^n$ and $B_i \in \mathrm{O}(n)$ for $i = 1, 2$, and  
let $E_i$ be the $(-1)$-eigenspace of $B_i$ restricted to $\mathrm{Span}(Z(\Nu))$ for $i = 1, 2$. 
Then $[\alpha_\sharp\Xi\beta] \in \psi^{-1}([\Gamma,\Nu])$,   
and $[\alpha_\sharp\Xi\beta]$ corresponds to the conjugacy class of the pair of elements 
$\{(\alpha\ov\gamma_1\alpha^{-1})_\star, (\alpha\ov\gamma_2\alpha^{-1})_\star\}$ of $\mathrm{Aff}(\Mu)$. 
Moreover $\psi^{-1}([\Gamma,\Nu])$ is the set of all the elements  of $\mathrm{Aff}(\Delta,\Mu)$ 
that correspond to the conjugacy class of the pair of elements 
$\{(\alpha\ov\gamma_1\alpha^{-1})_\star, (\alpha(v+\ov I)\ov\gamma_2\alpha^{-1})_\star\}$
of $\mathrm{Aff}(\Mu)$ for any $v \in E_1\cap E_2$. 
In particular, if $E_1\cap E_2 =\{0\}$, then $\psi^{-1}([\Gamma,\Nu]) = \{[\alpha_\sharp\Xi\beta]\}$. 
\end{theorem} 
\begin{proof}
By the surjective part of the proof of Lemma \ref{L:3}, 
we have that $[\alpha_\sharp\Xi\beta] \in \psi^{-1}([\Gamma,\Nu])$. 
For $i = 1,2$, we have that 
$$\alpha_\sharp \Xi \beta(\delta_i) =\alpha_\sharp\Xi(\Nu\gamma_i) = 
 \alpha_\sharp((\ov \gamma_i)_\star) = \alpha_\star(\ov \gamma_i)_\star\alpha_\star^{-1} 
 = (\alpha\ov\gamma_i\alpha^{-1})_\star.$$
Hence $[\alpha_\sharp\Xi\beta]$ corresponds to the conjugacy class of the pair 
$\{(\alpha\ov\gamma_1\alpha^{-1})_\star, (\alpha\ov\gamma_2\alpha^{-1})_\star\}$ by Lemma \ref{L:8}. 
Write $\alpha = a + A$ with $a \in E^{n-1}$ and $A: V \to E^{n-1}$ a linear isomorphism. 
Then for $i = 1, 2$, we have 
$$\alpha\ov \gamma_i\alpha^{-1} = (a+A)(\ov b_i + \ov B_i)(a+A)^{-1} = A\ov b_i + (I-A\ov B_iA^{-1})a + A\ov B_iA^{-1}.$$
The $(-1)$-eigenspace of $A\ov B_iA^{-1}$ restricted to $\mathrm{Span}(Z(\Mu))$ is $A(E_i)$ for $i = 1,2$. 

Suppose $v \in E_1\cap E_2$. Then  $Av \in A(E_1)\cap A(E_2)$ and 
$$\alpha(v+\ov I)\ov\gamma_2\alpha^{-1} = 
\alpha(v+\ov I)\alpha^{-1}\alpha\ov\gamma_2\alpha^{-1} = (Av + \ov I)\alpha\ov\gamma_2\alpha^{-1}.$$
Now, we have that 
\begin{eqnarray*}
((Av + \ov I)\alpha\ov\gamma_2\alpha^{-1})_\star^2 
& =  & ((Av + \ov I)\alpha\ov\gamma_2\alpha^{-1}(Av + \ov I)\alpha\ov\gamma_2\alpha^{-1})_\star \\
& =  & ((Av + \ov I)(-Av+\ov I) \alpha\ov\gamma_2\alpha^{-1} \alpha\ov\gamma_2\alpha^{-1})_\star 
\ \  =\ \  \ov I_\star.
\end{eqnarray*}
Hence, the order of  $((Av + \ov I)\alpha\ov\gamma_2\alpha^{-1})_\star$ is at most 2. 
Moreover the element
$$\Omega((\alpha\ov\gamma_1\alpha^{-1})_\star((Av + \ov I)\alpha\ov\gamma_2\alpha^{-1})_\star)$$
of $\mathrm{Out}(\Mu)$ has finite order, since 
$$\Omega((\alpha\ov\gamma_1\alpha^{-1})_\star((Av + \ov I)\alpha\ov\gamma_2\alpha^{-1})_\star) 
= \Omega((\alpha\ov\gamma_1\alpha^{-1})_\star(\alpha\ov\gamma_2\alpha^{-1})_\star).$$
Therefore, the conjugacy class of the pair of elements 
$\{(\alpha\ov\gamma_1\alpha^{-1})_\star, (\alpha(v+\ov I)\ov\gamma_2\alpha^{-1})_\star\}$
of $\mathrm{Aff}(\Mu)$ correspond to an element of  $\psi^{-1}([\Gamma,\Nu])$ by Lemma \ref{L:8} and Theorem \ref{T:9}. 
Thus  $\psi^{-1}([\Gamma,\Nu])$ is the set of all the elements of $\mathrm{Aff}(\Delta,\Mu)$ 
that correspond to the conjugacy class of the pair of elements 
$\{(\alpha\ov\gamma_1\alpha^{-1})_\star, (\alpha(v+\ov I)\ov\gamma_2\alpha^{-1})_\star\}$
of $\mathrm{Aff}(\Mu)$ for any $v \in E_1\cap E_2$ by Lemma \ref{L:8} and Theorem \ref{T:9}. 
\end{proof}

\noindent{\bf Example 1.}  
Let $e_1$ and $e_2$ be the standard basis vectors of $E^2$. 
Let $\Gamma$ be the group generated by $t_1 = e_1+I$ and $t_2= e_2+I$ and $-I$.  
Then $\Gamma$ is a 2-space group, and $E^2/\Gamma$ is a pillow. 
Let $\Nu=\langle t_1\rangle$.  
Then $\Nu$ is a complete normal subgroup of $\Gamma$, 
with $V = \mathrm{Span}(\Nu) = \mathrm{Span}\{e_1\}$.  
The quotient $\Gamma/\Nu$ is an infinite dihedral group 
generated by $\Nu t_2$ and $\Nu(-I)$. 
Let $\gamma_1 = -I$ and $\gamma_2 = e_2 - I$. 
Then $\Nu\gamma_1$ and $\Nu\gamma_2$ are Coxeter generators of $\Gamma/\Nu$. 

Let $\Delta$ be the standard infinite dihedral group, 
and let $\Mu$ be the standard infinite cyclic 1-space group generated by $\ov t_1 = e_1 + \ov I$. 
By Theorem \ref{T:10}, we have that $\psi^{-1}([\Gamma,\Nu])$ 
consists of all the elements $[\eta] \in \mathrm{Aff}(\Delta,\Mu)$ 
that correspond to the conjugacy class of the pair of elements 
$\{(\ov \gamma_1)_\star, ((v+\ov I)\ov \gamma_2)_\star\}$ of $\mathrm{Isom}(E^1/\Mu)$  for any $v \in E^1$. 
Here $\ov \gamma_1 =\ov \gamma_2 = -\ov I$. 

The reflections $(\ov \gamma_1)_\star$ and $((v+\ov I)\ov \gamma_1)_\star$ 
of the circle $E^1/\Mu$  lie in the same  
connected component $\mathcal{K}$ of the Lie group $\mathrm{Isom}(E^1/\Mu)$.  
Define a metric on $\mathcal{K}$ so that $\xi: E^1/\Mu \to \mathcal{K}$ defined by $\xi(\Mu v) = ((v+\ov I)\ov \gamma_1)_\star$ is an isometry. 
Conjugating by an element of $\mathrm{Isom}(E^1/\Mu)$ is an isometry of $\mathcal{K}$ with respect to this metric. 
Hence, the distance between $(\ov \gamma_1)_\star$ and $((v+\ov I)\ov \gamma_1)_\star$
is an invariant of the conjugacy class of the pair $\{(\ov \gamma_1)_\star, ((v+\ov I)\ov \gamma_1)_\star\}$. 
If $0\leq v \leq 1/2$, 
then the distance between $(\ov \gamma_1)_\star$ and $((v+\ov I)\ov \gamma_1)_\star$ is $v$,  
and so $\psi^{-1}([\Gamma,\Nu])$ has uncountably many elements. 

Let $v \in \realnos$ with $0 \leq v < 1$,  
and let $\Gamma_v$ be the group generated by $t_1 = e_1+I$ and $t_2 = ve_1+e_2+I$ and $-I$.  
Then $\Gamma_v$ is a 2-space group, and $E^2/\Gamma_v$ is a pillow. 
Let $\Nu=\langle t_1\rangle$.  
Then $\Nu$ is a complete normal subgroup of $\Gamma_v$, 
and $\Gamma_v/\Nu$ is  an infinite dihedral group. 
Let $\gamma_1 = -I$ and $\gamma_2 = ve_1+e_2-I$. 
Then $\Nu\gamma_1$ and $\Nu\gamma_2$ are Coxeter generators of $\Gamma_v/\Nu$. 
Let $\beta_v: \Delta \to \Gamma_v/\Nu$ be the isomorphism defined by $\beta_v(\delta_i) = \gamma_i$ for $i = 1,2$. 
Let $V = \mathrm{Span}(\Nu) = \mathrm{Span}\{e_1\}$, and 
let $\Xi_v: \Gamma_v/\Nu \to \mathrm{Isom}(V/\Nu)$ be the homomorphism induced 
by the action of $\Gamma_v/\Nu$ on $V/\Nu$. 
Then $[\Xi_v\beta_v] \in \psi^{-1}([\Gamma_v,\Nu])$ and $[\Xi_v\beta_v]$ 
corresponds to the pair of elements 
$$\{(\ov\gamma_1)_\star, (\ov\gamma_2)_\star\} = \{(\ov\gamma_1)_\star, ((v+I)\ov\gamma_1)_\star\}$$
of $\mathrm{Isom}(E^1/\Mu)$ by Theorem \ref{T:10}. 
Hence $[\Gamma_v,\Nu] = [\Gamma,\Nu]$ by Theorem \ref{T:9}. 

Let $\Kappa$ be the kernel of the action of $\Gamma_v$ on $V/\Nu$. 
The structure group $\Gamma_v/\Nu\Kappa$ is a dihedral group 
generated by $\Nu\Kappa t_2$ and $\Nu\Kappa (-I)$ ,  
The element $\Nu\Kappa t_2$ acts on the circle $V/\Nu$, of length one, 
by rotating a distance $v$, and  $\Nu\Kappa (-I)$ acts on $V/\Nu$ as a reflection. 
Let $c, d \in \integers$.  Then 
$t_1^ct_2^d = (c+dv)e_1+de_2 + I,$
and $t_1^ct_2^d \in \Kappa$ if and only if $c+dv = 0$.  
Thus if $v$ is irrational, then $\Kappa = \{I\}$, and $\Gamma_v/\Nu\Kappa$ is infinite.  
If $v = a/b$ with $a, b\in\integers$, $b> 0$, and $a, b$ coprime, 
then $\Nu\Kappa t_2$ has order $b$  in $\Gamma_v/\Nu\Kappa$, 
since $\Gamma_v/\Nu\Kappa$ acts effectively on $V/\Nu$ by Theorem 3(2) of \cite{R-T-C}. 
Thus $\Nu$ has an orthogonal dual in $\Gamma_v$ if and only if $v$ is rational. 
This example shows that  the order of the structure group $\Gamma_v/\Nu\Kappa$ 
is not necessarily an invariant of the affine equivalence class of $(\Gamma_v,\Nu)$. 
See also Example 3 of \cite{R-T-C}. 

\section{Classification from the Action of the Structure Group}\label{S:9} % 9

Let $\Nu$ be a normal subgroup of an $n$-space group $\Gamma$ 
such that $\Gamma/\Nu$ is either infinite cyclic or infinite dihedral. 
We say that a pair of generators $\{\Nu\gamma_1,\Nu\gamma_2\}$ of $\Gamma/\Nu$ 
is {\it canonical} when $\gamma_2=\gamma_1^{-1}$ if $\Gamma/\Nu$ is cyclic or 
$\Nu\gamma_1,\Nu\gamma_2$ are Coxeter generators if $\Gamma/\Nu$ is dihedral. 
The isomorphism class of $(\Gamma,\Nu)$ is determined by the action of 
a canonical pair of generators of $\Gamma/\Nu$ on $V/\Nu$. 
Let $\Kappa$ be the kernel of the action of $\Gamma$ on $V = \mathrm{Span}(\Nu)$. 
The action of $\Gamma/\Nu$ on $V/\Nu$ factors through the action of 
the structure group $\Gamma/\Nu\Kappa$ on $V/\Nu$. 
The question then arises as to when a pair of generators of $\Gamma/\Nu\Kappa$ 
is the image of a canonical pair of generators of $\Gamma/\Nu$ 
under the natural projection from $\Gamma/\Nu$ to $\Gamma/\Nu\Kappa$. 
In this section, we answer this question.  
This will enable us to determine the action of $\Gamma/\Nu$ on $V/\Nu$
from the diagonal action of the structure group $\Gamma/\Nu\Kappa$ on $V/\Nu \times V^\perp/\Kappa$. 

We have a short exact sequence
$$1 \to \Nu\Kappa/\Nu \hookrightarrow \Gamma/\Nu \to \Gamma/\Nu\Kappa \to 1.$$
By Lemma \ref{L:9} below for the dihedral case, 
every normal subgroup of $\Gamma/\Nu$ of infinite index is trivial. 
Hence, if the structure group $\Gamma/\Nu\Kappa$ is infinite, 
then $\Kappa \cong \Nu\Kappa/\Nu$ is trivial, 
and so $\Gamma/\Nu$ is the structure group. 

Suppose that $\Gamma/\Nu$ is cyclic and the structure group $\Gamma/\Nu\Kappa$ is finite. 
Then $\Gamma/\Nu\Kappa$ is finite cyclic of order $m$ for some positive integer $m$, 
since $\Gamma/\Nu\Kappa$ is a quotient of $\Gamma/\Nu$. 
Hence, the number of generators of $\Gamma/\Nu\Kappa$ is equal to the Euler phi function of $m$,  
and so $\Gamma/\Nu\Kappa$ may have more than two generators, 
and a generator of  $\Gamma/\Nu\Kappa$ may not lift to a generator of $\Gamma/\Nu$. 
The next theorem gives a necessary and sufficient condition for a generator of $\Gamma/\Nu\Kappa$ 
to lift to a generator of $\Gamma/\Nu$ with respect to the quotient map 
from $\Gamma/\Nu$ to $\Gamma/\Nu\Kappa$. 

\begin{theorem}\label{T:11} % 11
Let $\Nu$ be a normal subgroup of an $n$-space group $\Gamma$ 
such that $\Gamma/\Nu$ is infinite cyclic, 
and let $\Kappa$ be the kernel of the action of $\Gamma$ on $V = \mathrm{Span}(\Nu)$, 
and suppose that the structure group $\Gamma/\Nu\Kappa$ is finite of order $m$. 
Let $\gamma$ be an element of $\Gamma$ such that $\gamma\Nu\Kappa$ generates $\Gamma/\Nu\Kappa$.  
Then there exists an element $\delta$ of $\Gamma$ 
such that $\gamma\Nu\Kappa = \delta\Nu\Kappa$ and $\delta\Nu$ generates $\Gamma/\Nu$ 
if and only if $\gamma\Nu\Kappa$ acts on the circle $V^\perp/\Kappa$ 
by a rotation of $\pm2\pi/m$. 
\end{theorem}
\begin{proof}
Suppose $\delta\Nu$ is a generator of $\Gamma/\Nu$. 
Since $V^\perp/(\Gamma/\Nu)$ is a circle, $\delta$ acts as a translation $d+I$ on $V^\perp$ with $d\neq 0$. 
As $\Nu\Kappa/\Nu$ is a subgroup of $\Gamma/\Nu$ of index $m$ and $\Kappa \cong \Nu\Kappa/\Nu$ is infinite cyclic,  
the group $\Kappa$ has a generator that acts as a translation $md+I$ on $V^\perp$. 
Therefore $\delta\Nu\Kappa$ acts by a rotation of $\pm2\pi/m$ on the circle $V^\perp/\Kappa$. 

Suppose $\gamma$ is an element of $\Gamma$ such that  
$\gamma\Nu\Kappa$ acts by a rotation of $\pm2\pi/m$ on the circle $V^\perp/\Kappa$. 
Then $\gamma\Nu\Kappa = \delta^{\pm 1}\Nu\Kappa$, since 
the group $\Gamma/\Nu\Kappa$ acts effectively on $V^\perp/\Kappa$ 
by Theorem 3(2) of \cite{R-T-C}.
\end{proof} 

\begin{lemma}\label{L:9}  % 9
Let $\Delta$ be an infinite dihedral group with Coxeter generators $\alpha, \beta$. 
Then the proper normal subgroups of $\Delta$ are the infinite dihedral groups $\langle \alpha, \beta\alpha\beta\rangle$ 
and $\langle \beta,\alpha\beta\alpha\rangle$ of index 2, and the infinite cyclic group $\langle (\alpha\beta)^m\rangle$ 
of index $2m$ for each positive integer $m$. 
\end{lemma}
\begin{proof}
We may assume that $\Delta$ is a discrete group of isometries of $E^1$. 
Then every element of $\Delta$ is either a translation or a reflection. 
Let $\Nu$ be a proper normal subgroup of $\Delta$. 
Suppose $\Nu$ contains a reflection.  Then $\Nu$ contains either $\alpha$ or $\beta$, 
since every reflection in $\Delta$ is conjugate in $\Delta$ to either $\alpha$ or $\beta$. 
Let $\langle\langle \alpha\rangle\rangle$ be the normal closure of $\langle \alpha\rangle$ in $\Delta$. 
Then $\langle\langle \alpha\rangle\rangle$ is the infinite dihedral group $ \langle \alpha, \beta\alpha\beta\rangle$. 
Let  $\langle\langle\beta\rangle\rangle$ be the normal closure of $\langle \beta\rangle$ in $\Delta$. 
Then $\langle\langle\beta\rangle\rangle$ is the infinite dihedral group $ \langle \beta, \alpha\beta\alpha\rangle$. 
Now $\Nu$ contains either $\langle\langle \alpha\rangle\rangle$ or $\langle\langle\beta\rangle\rangle$. 
As both $\langle\langle \alpha\rangle\rangle$ and $\langle\langle\beta\rangle\rangle$
have index 2 in $\Delta$, we have that $\Nu$ is either  $\langle\langle \alpha\rangle\rangle$ or $\langle\langle\beta\rangle\rangle$. 

Now suppose $\Nu$ does not contain a reflection.  
Then $\Nu$ is a subgroup of the group $\langle \alpha\beta\rangle$ of translations of $\Delta$.  
 Hence $\Nu = \langle (\alpha\beta)^m\rangle$ for some positive integer $m$. 
 Moreover each $m$ is possible, since $\langle \alpha\beta\rangle$ is a characteristic subgroup of $\Delta$. 
 \end{proof}

Now assume that $\Gamma/\Nu$ is dihedral and $\Gamma/\Nu\Kappa$ is finite. 
Then $\Gamma/\Nu\Kappa$ is trivial or a finite dihedral group of order $2m$ for some positive integer $m$ 
by Lemma \ref{L:9}, since $\Gamma/\Nu\Kappa$ is a quotient of $\Gamma/\Nu$. 
The next theorem gives necessary and sufficient conditions for a pair of generators of $\Gamma/\Nu\Kappa$ 
to lift to a pair of Coxeter generators of $\Gamma/\Nu$ with respect to the the quotient map 
from $\Gamma/\Nu$ to $\Gamma/\Nu\Kappa$.

 \begin{theorem}\label{T:12} % 12
Let $\Nu$ be a normal subgroup of an $n$-space group $\Gamma$ 
such that $\Gamma/\Nu$ is infinite dihedral, 
and let $\Kappa$ be the kernel of the action of $\Gamma$ on $V = \mathrm{Span}(\Nu)$, 
and suppose that the structure group $\Gamma/\Nu\Kappa$ is finite. 
Let $\gamma_1$ and $\gamma_1$ be elements of $\Gamma$ such that $\{\gamma_1\Nu\Kappa,\gamma_2\Nu\Kappa\}$ 
generates $\Gamma/\Nu\Kappa$. 
Then there exists elements $\delta_1$ and $\delta_2$ of $\Gamma$ 
such that $\{\gamma_1\Nu\Kappa, \gamma_2\Nu\Kappa\} = \{\delta_1\Nu\Kappa, \delta_2\Nu\Kappa\}$,  
and $\{\delta_1\Nu,\delta_2\Nu\}$ is a set of Coxeter generators of $\Gamma/\Nu$  
if and only if either
\begin{enumerate}
\item The order of $\Gamma/\Nu\Kappa$ is 1, or
\item The order of $\Gamma/\Nu\Kappa$ is 2, the group $\Kappa$ is infinite dihedral,  
and one of $\gamma_1\Nu\Kappa$ or $\gamma_2\Nu\Kappa$ is the identity element of $\Gamma/\Nu\Kappa$,  
and the other acts as the reflection of the closed interval $V^\perp/\Kappa$, or 
\item The order of $\Gamma/\Nu\Kappa$ is $2m$ for some positive integer $m$, 
the group $\Kappa$ is infinite cyclic, and both $\gamma_1\Nu\Kappa$ and $\gamma_2\Nu\Kappa$ 
act as reflections of  the circle $V^\perp/\Kappa$, and $\gamma_1\gamma_2\Nu\Kappa$ acts on $V^\perp/\Kappa$ 
as a rotation of $\pm2\pi/m$. 
\end{enumerate}
\end{theorem}
\begin{proof}
Suppose that $\delta_1$ and $\delta_2$ are elements of $\Gamma$ 
such that $\{\delta_1\Nu,\delta_2\Nu\}$ is a set of Coxeter generators of $\Gamma/\Nu$ 
and $\{\gamma_1\Nu\Kappa, \gamma_2\Nu\Kappa\} = \{\delta_1\Nu\Kappa, \delta_2\Nu\Kappa\}$. 
The group $\Nu\Kappa/\Nu$ is a normal subgroup of $\Gamma/\Nu$ of finite index, 
since $(\Gamma/\Nu)/(\Nu\Kappa/\Nu) = \Gamma/\Nu\Kappa$. 
We have that $\Nu\Kappa/\Nu \cong \Kappa$, since $\Nu\cap\Kappa =\{I\}$.   
Suppose that $\Kappa$ is infinite dihedral. 
Then the order of $\Gamma/\Nu\Kappa$ is 1 or 2  by Lemma \ref{L:9}. 
If the order of $\Gamma/\Nu\Kappa$ is 2, 
then $\Nu\Kappa/\Nu$ is either $\langle \delta_1\Nu, \delta_2\delta_1\delta_2\Nu\rangle$ 
or $\langle \delta_2\Nu, \delta_1\delta_2\delta_1\Nu\rangle$ by Lemma \ref{L:9}, and so 
either  $\delta_1\Nu\Kappa = \Nu\Kappa$ or  $\delta_2\Nu\Kappa = \Nu\Kappa$, and 
hence one of  $\gamma_1\Nu\Kappa$  or $\gamma_2\Nu\Kappa$ is  $\Nu\Kappa$ 
and the other acts as the reflection of the closed interval $V^\perp/\Kappa$, 
since $\Gamma/\Nu\Kappa$ acts effectively on $V^\perp/\Kappa$ by Theorem 3(2) of \cite{R-T-C}.

Suppose that $\Kappa$ is infinite cyclic. 
Then the order of $\Gamma/\Nu\Kappa$ is $2m$ for some positive integer $m$ 
and $\Nu\Kappa/\Nu = \langle (\delta_1\delta_2)^m\Nu\rangle$ by Lemma \ref{L:9}. 
Hence both $\delta_1\Nu\Kappa$ and $\delta_2\Nu\Kappa$ 
act as reflections of  the circle $V^\perp/(\Nu\Kappa/\Nu) = V^\perp/\Kappa$,  
and $\delta_1\delta_2\Nu\Kappa$ acts on $V^\perp/\Kappa$ 
as a rotation of $\pm2\pi/m$. 

Conversely, suppose $\delta_1,\delta_2$ are elements of $\Gamma$ such that 
$\{\delta_1\Nu,\delta_2\Nu\}$ is a set of Coxeter generators of $\Gamma/\Nu$. 
If $\Gamma/\Nu\Kappa$ is trivial, then obviously 
$\{\gamma_1\Nu\Kappa, \gamma_2\Nu\Kappa\} = \{\delta_1\Nu\Kappa, \delta_2\Nu\Kappa\}$. 
Suppose next that statement (2) holds.  
Then $\Nu\Kappa/\Nu$ is an infinite dihedral group, 
and either $\delta_1\Nu\Kappa$ or $\delta_2\Nu\Kappa$ is trivial in 
$(\Gamma/\Nu)/(\Nu\Kappa/\Nu) = \Gamma/\Nu\Kappa$ by Lemma \ref{L:9}. 
Hence $\{\gamma_1\Nu\Kappa,\gamma_2\Nu\Kappa\} = \{\delta_1\Nu\Kappa,\delta_2\Nu\Kappa\}$, 
since $\Gamma/\Nu\Kappa$ has order  2. 

Now suppose statement (3) holds. 
Then $\Nu\Kappa/\Nu = \langle (\delta_1\delta_2)^m\Nu\rangle$ by Lemma \ref{L:9}. 
Hence, a generator of the infinite cyclic group $\Kappa$ acts on the line $V^\perp$ 
in the same way that $(\delta_1\delta_2)^m\Nu$ acts on $V^\perp$ as a translation. 
Therefore $\delta_1\delta_2\Nu\Kappa$ acts as a rotation of $\pm2\pi/m$ on the circle $V^\perp/\Kappa$.  
Hence $\gamma_1\gamma_2\Nu\Kappa = (\delta_1\delta_2)^{\pm 1}\Nu\Kappa$, 
since $\Gamma/\Nu\Kappa$ acts effectively on $V^\perp/\Kappa$. 
Therefore, there exists $\gamma\in\Gamma$ such that 
$\{\gamma_1\Nu\Kappa,\gamma_2\Nu\Kappa\} = 
\{\gamma\delta_1\gamma^{-1}\Nu\Kappa,\gamma\delta_2\gamma^{-1}\Nu\Kappa\}$.  
Moreover $\{\gamma\delta_1\gamma^{-1}\Nu, \gamma\delta_2\gamma^{-1}\Nu\}$ 
is a set of Coxeter generators of $\Gamma/\Nu$.
\end{proof}

%\clearpage

%%%%%%%%%%%%%%%%%%%%%%%%%%%%%%%%%%

\section{The Classification of  Geometric Fibrations of Flat 2-Orbifolds}\label{S:10}  % 10

For 2-dimensional orbifolds, Seifert fibrations and co-Seifert fibrations are the same. 
In this section, we describe the classification of all the Seifert geometric fibrations of compact, connected, flat 2-orbifolds up to affine equivalence. 
We denote a circle by $\mathrm{O}$ and a closed interval by $\mathrm{I}$.

Table 1 describes, via the generalized Calabi construction,  
all the Seifert and dual Seifert fibrations of a compact, connected, flat 2-orbifold up to affine equivalence.  

(1) The first column lists the IT number of the corresponding 2-space group $\Gamma$ 
given in Table 1A of \cite{B-Z}. Only 2-space groups with 
IT numbers 1 - 9 appear in Table 1 by Theorem 11 of \cite{R-T}. 

(2) The second column lists the Conway name \cite{Conway} of the corresponding flat orbifold $E^2/\Gamma$. 

(3)  The third column lists the fiber $V/\Nu$ and base $V^\perp/(\Gamma/\Nu)$ 
of the Seifert fibration corresponding to a 1-dimensional, complete, normal subgroup $\Nu$ of $\Gamma$ 
with $V = \mathrm{Span}(\Nu)$. 
Parentheses indicates that the fiber is $\mathrm{O}$, 
and closed brackets indicates that the fiber is $\mathrm{I}$. 
A dot indicates that the base is $\mathrm{O}$ and a dash indicates 
that the base is $\mathrm{I}$. 
For example, $(\,\cdot\,)$ indicates a fibration with fiber and base $\mathrm{O}$. 
The group $\Nu$ is described in \S 10 of \cite{R-T}. 
For the first two rows of Table 1, the group $\Nu$ corresponds to the parameters $(a,b) = (1,0)$ 
in cases (1) and (2) of \S 10 of \cite{R-T}. 

(4) The fourth column indicates whether or not the corresponding 
space group extension $1\to \Nu\to \Gamma\to \Gamma/\Nu\to 1$ splits 
as implied by Theorem 14 of \cite{R-T-C}. 

(5) The fifth column lists the fiber $V^\perp/\Kappa$ and base $V/(\Gamma/\Kappa)$, 
with $\Kappa = \Nu^\perp$,  of the dual Seifert fibration. 

(6) The sixth column indicates whether or not the corresponding 
space group extension $1\to \Kappa\to \Gamma\to \Gamma/\Kappa\to 1$ splits. 

(7) The seventh column lists the isomorphism type of the structure group $\Gamma/\Nu\Kappa$ with $C_n$ indicating a cyclic group of order $n$,  
and $D_2$ indicating a dihedral group of order 4. 
The order of the structure groups for the 2-space groups with IT numbers 1 and 2 in Table 1 
were chosen to be as small as possible. 
See Example 3 in \cite{R-T-C} and Example 1 for the full range of structure groups for  2-space groups with IT numbers 1 and 2. 
The isomorphism types of the structure groups for the remaining 2-space groups are unique, 
since $\Nu$ and $\Kappa = \Nu^\perp$ are the only proper complete normal subgroups of $\Gamma$ in Rows 3 - 9. 

(8) The last column indicates how the structure group 
$\Gamma/\Nu\Kappa$ acts diagonally on the Cartesian product of the fibers $V/\Nu \times V^\perp/\Kappa$. 
We denote the identity map by idt., a halfturn by $2$-rot., and a reflection by ref. 
We denote the reflection of $\mathrm{O}$ orthogonal to ref.\ by ref.$'$. 

%\vspace{.15in}
\begin{table}  % 1
\begin{tabular}{llllllll}
no.  & CN & fibr. & split & dual & split & grp. & structure group action \\
\hline 
1 & $\circ$   & $(\,\cdot\,)$ & Yes & $(\,\cdot\, )$ & Yes & $C_1$ & (idt., idt.)  \\
2 & $2222$ & $(-)$            & Yes & $(-)$          & Yes & $C_2$ & (ref., ref.)\\
3 & $**$       & $(-)$            & Yes & $[\,\cdot\,]$& Yes & $C_1$ & (idt., idt.)\\
4 & $\times\times$ & $(-)$ & No & $(\,\cdot\, )$ & Yes & $C_2$ & (2-rot., ref.)\\
5 & $*\times$  & $(-)$        & No   & $[\,\cdot\,]$ & Yes & $C_2$ & (2-rot, ref.)\\
6 & $*2222$ & $[-]$            & Yes & $[-]$          & Yes & $C_1$ & (idt., idt.)\\
7 & $22*$   & $(-)$             & Yes  & $[-]$         & Yes   &  $C_2$ & (ref., ref.) \\
8 & $22\times$ & $(-)$      & No   & $(-)$         & No    &  $D_2$ &(ref., ref.), (2-rot., ref.$'$) \\
9 & $2\hbox{$*$}22$ &  $[-]$  & Yes & $[-]$       & Yes  & $C_2$ & (ref., ref.)
\end{tabular}

\vspace{.15in}
\caption{The Seifert and dual Seifert fibrations of the 2-space groups.}
\end{table}

The actions of the structure group on the fibers of the Seifert fibrations  
are described by Theorem 7 of \cite{R-T-C} and Example 6 in \cite{R-T-C}, 
except for the case when the structure group is a dihedral group of order 4. 
The problem with dihedral groups of order 4 is that it is not clear a priori 
which of the three nonidentity elements acts as a halfturn on a circle factor of $V/\Nu \times V^\perp/\Kappa$. 
See Example 3 for the description of the action in the case of Row 8 of Table 1.

\medskip
\noindent{\bf Example 2.} 
Let $\Gamma$ be the group in Table 1 with IT number 4. 
Then $E^2/\Gamma$ is a Klein bottle.  The structure group $G = \Gamma/\Nu\Kappa$ 
has order 2. The fibers $V/\Nu$ and $V^\perp/\Kappa$ are both circles. 
The generator of the group $G$ acts by a halfturn on $V/\Nu$, since $(V/\Nu)/G =V/(\Gamma/\Kappa)$ is a circle. 
The generator of the group $G$ acts by a reflection on $V^\perp/\Kappa$, 
since $(V^\perp/\Kappa)/G=V^\perp/(\Gamma/\Nu)$ is a closed interval. 
The action of the generator of $G$ on $V/\Nu\times V^\perp/\Kappa$ 
is specified by the entry (2-rot., ref.) in Row 4 of Table 1.

\medskip
\noindent{\bf Example 3.} 
Let $\Gamma$ be the group with IT number 8 in Table 1A of \cite{B-Z}. 
Then $\Gamma = \langle t_1, t_2, A, \beta\rangle$ 
where $t_i = e_i+I$ for $i=1,2$ are the standard translations, 
and $\beta = \frac{1}{2}e_1+\frac{1}{2}e_2 + B$, and 
$$A = \left(\begin{array}{rr} -1 & 0 \\ 0 & -1 \end{array}\right), \  \hbox{and}\ \ 
B = \left(\begin{array}{rr} 1 & 0 \\ 0 & -1 \end{array}\right).$$
The isomorphism type of $\Gamma$ is $22\times$ in Conway's notation or $pgg$ in IT notation.  
The orbifold $E^2/\Gamma$ is a projective pillow. 
The group $\Nu = \langle t_1\rangle$ is a complete normal subgroup of $\Gamma$, 
with $V= \mathrm{Span}(\Nu) = \mathrm{Span}\{e_1\}$. 
The flat orbifold $V/\Nu$ is a circle. 
Let $\Kappa = \Nu^\perp = \langle t_2\rangle$. 
Then $V^\perp/\Kappa$ is also a circle. 
The structure group $\Gamma/\Nu\Kappa$ is a dihedral group of order 4 
generated by $\Nu\Kappa A$ and $\Nu\Kappa \beta$. 
The element $\Nu\Kappa A$ acts as a reflection on $V/\Nu$ and on $V^\perp/\Kappa$.  
The action of $\Nu\Kappa A$ on $V/\Nu\times V^\perp/\Kappa$ is specified by the entry (ref., ref.) in Row 8 of Table 1. 
The element $\Nu\Kappa \beta$ acts on $V/\Nu$ as a halfturn and on $V^\perp/\Kappa$ as a reflection. 
The action of $\Nu\Kappa \beta$ on $V/\Nu\times V^\perp/\Kappa$ is specified by the entry (2-rot., ref.$'$) in Row 8 of Table 1.  

\vspace{.15in}
As discussed in \S 10 of \cite{R-T}, the fibration and dual fibration in Rows 1, 2, 6, 8, 9 of Table 1 
are affinely equivalent.  
No other pair of fibrations in Table 1, with the same fibers and the same bases,  
are affinely equivalent, 
since the corresponding 2-space groups are nonisomorphic. 
We next apply the theory for classifying co-Seifert geometric fibrations in \S 7 and \S 8 
to prove that every affine equivalence class of a Seifert geometric fibration of a compact, connected, flat 2-orbifold 
is represented by one of the fibrations described in Table 1. 

For simplicity, suppose that $\Delta$ and $\Mu$ are standard 1-space groups, $E^1/\langle e_1 + I\rangle$ 
or $E^1/\langle e_1+\nolinebreak I, -I\rangle$. 
We next describe $\mathrm{Iso}(\Delta,\Mu)$. 
Now $E^1/\Mu= \mathrm{O}$ or $\mathrm{I}$. 
The Lie group $\mathrm{Isom}(\mathrm{O})$ is isomorphic to $\mathrm{O}(2)$, 
and $\mathrm{Isom}(\mathrm{I})$ is a group of order 2 generated by the reflection ref.\ of $\mathrm{I}$ about its midpoint. 
Moreover $\mathrm{Aff}(\mathrm{O}) = \mathrm{Isom}(\mathrm{O})$ and 
$\mathrm{Aff}(\mathrm{I}) = \mathrm{Isom}(\mathrm{I})$, 
since length-preserving affinities of $E^1$ are isometries. 
Hence, for both isomorphism types of $\Mu$, the group $\mathrm{Out}(\Mu)$ has order 2 by Theorem 3 of \cite{R-T-I}. 
We represent  $\mathrm{Out}(\Mu)$ by the subgroup \{idt., ref.\} of $\mathrm{Isom}(E^1/\Mu)$ 
that is mapping isomorphically onto $\mathrm{Out}(\Mu)$ by $\Omega: \mathrm{Isom}(E^1/\Mu) \to \mathrm{Out}(\Mu)$. 

Assume first that $\Delta$ is infinite cyclic. 
The set $\mathrm{Iso}(\Delta,\Mu)$ consists of two elements 
corresponding to the pairs of inverse elements $\{$idt., idt.$\}$ and $\{$ref., ref.$\}$ 
of $\mathrm{Isom}(\mathrm{O})$ by Theorem \ref{T:7}. 
Thus, there are two affine equivalence classes of fibrations of type $(\,\cdot\,)$ 
corresponding to the case that $\Mu$ is infinite cyclic,  
and two affine equivalence classes of fibrations of type $[\,\cdot\,]$ 
corresponding to the case that $\Mu$ is infinite dihedral. 
The fibration of type $(\,\cdot\,)$,  corresponding to $\{$idt., idt.$\}$, $\{$ref., ref.$\}$, 
is described in Row 1, 4, respectively, of Table 1 by Theorem \ref{T:11}.  
The fibration of type $[\,\cdot\,]$, corresponding to $\{$idt., idt.$\}$, $\{$ref., ref.$\}$, 
is described in Row 3, 5, respectively, of Table 1 by Theorem \ref{T:11}. 

Now assume that both $\Delta$ and $\Mu$ are infinite dihedral. 
Then the set $\mathrm{Iso}(\Delta,\Mu)$ consists of three elements 
corresponding to the pairs of elements $\{\mathrm{idt.},\mathrm{idt.}\}$, 
$\{\mathrm{idt.}, \mathrm{ref.}\}$, $\{\mathrm{ref.}, \mathrm{ref.}\}$ of $\mathrm{Isom}(\mathrm{I}$) by Theorem \ref{T:8}. 
Thus, there are three affine equivalence classes of fibrations of type $[ - ]$. 
The corresponding fibration of type $[ - ]$ is described in Row 6, 9, 7, respectively, of Table 1 by Theorem \ref{T:12}. 

Finally, assume that $\Delta$ is infinite dihedral and $\Mu$ is infinite cyclic. 
There are two conjugacy classes of isometries of $\mathrm{O}$ of order 2, 
the class of the halfturn 2-rot.\ of $\mathrm{O}$ and the class of a reflection ref.\ of $\mathrm{O}$. 
By Lemma \ref{L:8} and Theorem \ref{T:10}, the set $\mathrm{Iso}(\Delta,\Mu)$ consists of six elements 
corresponding to the pairs of elements 
\{idt., idt.\}, \{idt., 2-rot.\}, \{idt., ref.\}, \{2-rot., 2-rot.\}, \{2-rot., ref.\}, \{ref., ref.\} of $\mathrm{Isom}(\mathrm{O})$.  
Only the pair \{ref., ref.\} falls into the case $E_1\cap E_2 \neq \{0\}$ of Theorem \ref{T:10}. 
Thus, there are six affine equivalence classes of fibrations of type $( - )$. 
The corresponding fibration of type $( - )$ is described in Row 3, 5, 7, 4, 8, 2, respectively, of Table 1 by Theorem \ref{T:12}. 
Thus, every affine equivalence class of Seifert geometric fibrations of a compact, connected, flat 2-orbifold is represented by one of the fibrations in Table 1. 
Our last theorem summarizes the classification of Seifert (or co-Seifert) geometric fibrations 
of compact, connected, flat 2-orbifolds. 

\begin{theorem}\label{T:13} % 13 
There are exactly 13 affine equivalence classes of Seifert geometric fibrations 
of compact, connected, flat 2-orbifolds. There are two classes of type $(\,\cdot\,)$, 
two classes of type $[\,\cdot\,]$, three classes of type $[ - ]$, and six classes of type $( - )$. 
\end{theorem}

\end{document}